\documentclass[11pt]{amsart}

\usepackage[a4paper,hmargin=3cm,vmargin=3cm]{geometry}
\usepackage{amsfonts,amssymb,amscd,amstext,hyperref}
\usepackage{graphicx}
\usepackage{color}
\usepackage[dvips]{epsfig}
\usepackage[latin1]{inputenc}

\usepackage{fancyhdr}
\usepackage{times}

\usepackage{enumerate}
\usepackage{titlesec}
\usepackage{mathrsfs}
\usepackage{stmaryrd}

\def\R{\mathbb{R}}

\newcommand{\ben}{\begin{enumerate}}
\newcommand{\bit}{\begin{itemize}}
\newcommand{\een}{\end{enumerate}}
\newcommand{\eit}{\end{itemize}}

\newcommand{\ed}{\end{document}}

\def\cR{\mathcal{R}}
\def\cB{\mathcal{B}}
\def\cW{\mathcal{W}}

\def\cG{\mathcal{G}}
\def\cF{\mathcal{F}}
\def\cN{\mathcal{N}}

\def\cO{\mathcal{O}}

\let\8=\infty \let\0=\emptyset

\let\landa=\lambda
\let\alfa=\alpha

\let\parc=\partial

\def\ep{\varepsilon}

\def\landa{\lambda}

\def\flecha{\rightarrow}
\def\esiz{\langle}
\def\esde{\rangle}

\def\S{\Sigma}

\def\cte.{\mathop{\rm cte.}\nolimits}

\def\R{\mathbb{R}}

\def\S{\mathbb{S}}

\newfont{\bb}{msbm10 at 12pt}

\titleformat{\section}
{\filcenter\bfseries\large} {\thesection{.}}{0.2cm}{}
\titleformat{\subsection}[runin]
{\bfseries} {\thesubsection{.}}{0.15cm}{}[.]
\titleformat{\subsubsection}[runin]
{\em}{\thesubsubsection{.}}{0.15cm}{}[.]

\usepackage[up,bf]{caption}

\newtheorem{theorem}{Theorem}[section]
\newtheorem{lemma}[theorem]{Lemma}
\newtheorem{proposition}[theorem]{Proposition}

\newtheorem{remark}[theorem]{Remark}
\newtheorem{corollary}[theorem]{Corollary}
\newtheorem{definition}[theorem]{Definition}

\theoremstyle{definition}

\numberwithin{equation}{section}
\numberwithin{figure}{section}

\usepackage{color}

\begin{document}
\fancyhead[LO]{Elliptic Weingarten surfaces}
\fancyhead[RE]{Isabel Fernández, Pablo Mira}
\fancyhead[RO,LE]{\thepage}

\thispagestyle{empty}

\begin{center}
{\bf \LARGE Elliptic Weingarten surfaces: singularities,\\[0.2cm] rotational examples and the halfspace theorem}
\vspace*{5mm}

\hspace{0.2cm} {\Large Isabel Fernández and Pablo Mira}
\end{center}

%



\footnote[0]{
\noindent \emph{Mathematics Subject Classification}: 53A10, 53C42, 35J15, 35J60. \\ \mbox{} \hspace{0.25cm} \emph{Keywords}: Weingarten surfaces, fully nonlinear elliptic equations, phase space analysis, halfspace theorem, isolated singularities, rotational surfaces.}



\vspace*{7mm}

\begin{quote}
{\small
\noindent {\bf Abstract}\hspace*{0.1cm}
We show by phase space analysis that there are exactly $17$ possible qualitative behaviors for a rotational surface in $\R^3$ that satisfies an arbitrary elliptic Weingarten equation $W(\kappa_1,\kappa_2)=0$, and study the singularities of such examples. As global applications of this classification, we prove a sharp halfspace theorem for general elliptic Weingarten equations of finite order, and a classification of peaked elliptic Weingarten spheres with at most $2$ singularities. In the case that $W$ is not elliptic, we give a negative answer to a question by Yau regarding the uniqueness of rotational ellipsoids.

\vspace*{0.1cm}

}
\end{quote}

%
%
%
%
%
%
%
%


\section{Introduction}

An immersed oriented surface $\Sigma$ in $\R^3$ is called a \emph{Weingarten surface} if its principal curvatures $\kappa_1,\kappa_2$ satisfy a smooth relation 
 \begin{equation}\label{eqw}
 W(\kappa_1,\kappa_2)=0,
 \end{equation} 
for some $W\in C^1(\R^2)$. The relation \eqref{eqw} defines a fully nonlinear PDE when we view $\Sigma$ as a local graph, and we say that the Weingarten equation \eqref{eqw} is \emph{elliptic} if this equation is elliptic. There are several ways of writing \eqref{eqw} in the elliptic case, which will be discussed in detail in Section \ref{sec:prelim}. The most studied elliptic Weingarten surfaces are, of course, constant mean curvature (CMC) surfaces in $\R^3$, and minimal surfaces in particular. Another important example are the surfaces of positive constant Gaussian curvature (CGC).

Elliptic Weingarten surfaces in $\R^3$ represent the natural \emph{fully nonlinear} extension of CMC surface theory. In this sense, it has been an active research field to explore which global results of CMC surface theory extend to the general case of elliptic Weingarten surfaces. Some contributions to this topic can be found in \cite{AEG,A,B,Ch,CF,EM,FGM,GMM,GM3,HW1,Ho0,Ho,JS,rafa,RS,SaTo,SaTo2,T}.

However, there are considerable differences between the geometry of elliptic Weingarten surfaces and CMC surfaces. A fundamental one is that elliptic Weingarten surfaces often present singularities, which can be of very different types. As an example, a parallel surface of a regular elliptic Weingarten surface is also an elliptic Weingarten surface (for a different equation), but in the process some singularities on the surface can be formed. Also, elliptic Weingarten surfaces sometimes admit isolated embedded \emph{conical} singularities (see e.g. \cite{GHM,GJM}), a behavior not possible for CMC surfaces in $\R^3$. In this sense, it is a challenging problem to understand the behavior of singularities of elliptic Weingarten surfaces. 

A main difficulty for the study of such singularities is the lack of examples, which is due to the fully nonlinear nature of the theory, and the unavailability of the usual CMC methods. With this motivation, our objective in this paper is to understand all possible singularities of rotational elliptic Weingarten surfaces, and to extract some general global applications from their study. 

Our study can be seen as a natural continuation of results by R. Sa Earp and E. Toubiana \cite{SaTo,SaTo2}. In these works, they gave a classification of all \emph{complete} rotational elliptic Weingarten surfaces, using a method different from the one that will be used here. These complete examples have been key tools in the development of the global theory of elliptic Weingarten surfaces, see e.g. \cite{EM,RS,SaTo2}. Sa Earp and Toubiana also provided some examples of rotational elliptic Weingarten surfaces that are not complete, due to the existence of singularities; see Theorems 3 and 4 in \cite{SaTo} and Theorem 3 in \cite{SaTo2}. The discussion in the present paper classifies both regular and singular examples, thus completing the work by Sa Earp and Toubiana to a full classification. Previous existence or classification theorems for some particular theories of elliptic Weingarten surfaces (e.g., $a\kappa_1 + b\kappa_2 =c$ with $ab>0$), with or without singularities, appear in \cite{rafa,MO2,P,JS,RS}.

We next describe our results and the organization of the paper.

In Section \ref{sec:prelim} we rewrite the elliptic Weingarten equation in a way that easens the classification result. Specifically, we will view it as a relation $\kappa_2=g(\kappa_1)$, with $g'<0$, between the principal curvatures $\kappa_1\geq \kappa_2$ of the surface. Here, the condition $g'<0$ comes from ellipticity. We will denote by $\cW_g$ the class of all elliptic Weingarten surfaces associated with a particular choice of such a function $g$.

In Section \ref{sec:fases} we develop a phase space analysis for the general study of rotational surfaces of an arbitrary elliptic Weingarten class $\cW_g$. This theoretical tool, inspired in part by \cite{BGM}, also seems adequate for treating the case of non-elliptic Weingarten surfaces. For some works about rotational non-elliptic Weingarten surfaces, see e.g. \cite{BO,CC,KS,rafa,MO1,MO2}.

In Section \ref{sec:existence} we will determine which elliptic Weingarten classes $\cW_g$ admit rotational examples with singularities, and we will describe the geometry of such singularities whenever they occur (Theorem \ref{th:existence}). We will show that if the Weingarten equation is uniformly elliptic, any rotational example within the class $\cW_g$ is complete and regular (Corollary \ref{co:uniform}).

In Section \ref{sec:classification} we will describe the moduli space of all rotational elliptic Weingarten surfaces, in terms of natural geometric parameters. We show that there are exactly 17 possible qualitative behaviors for such surfaces, 10 of them with singularities. Our classification also indicates, for each particular elliptic Weingarten class $\cW_g$, which specific examples appear. See Theorems \ref{th:Kpositiva}, \ref{th:nonminimal}, \ref{th:minimal} and \ref{heica}. See also Figures \ref{fig:esfera_sing}, \ref{fig:unduloideNodoide_sing}, \ref{fig:catenoide_sing} and Subsection \ref{sec:summary} as a summary and visualization of the classification theorem.

In Section \ref{sec:applications} we use the classification in Section \ref{sec:classification} to study the validity of the halfspace theorem for elliptic Weingarten surfaces. The famous Hoffman-Meeks halfspace theorem \cite{HM} shows that any properly immersed minimal surface contained in a halfspace of $\R^3$ must be a plane. Sa Earp and Toubiana gave counterexamples to this property in the more general situation of elliptic Weingarten surfaces with $g(0)=0$. They also gave sufficient conditions on $g$ under which an elliptic Weingarten class $\cW_g$ satisfies the halfspace theorem. Using rotational surfaces with singularities as comparison objects, we will sharpen these results. We will show that the conditions $g(0)=0$ and $g'(0)=-1$ are necessary for the validity of the halfspace property for the ellipitc Weingarten class $\cW_g$ (Corollary \ref{neces}), and that if $g$ has finite order at $0$, these conditions are also sufficient (Theorem \ref{sufici} and Corollary \ref{th:order}).  

In Section \ref{sec:isolated} we study isolated singularities. We prove that any elliptic Weingarten graph defined on a punctured disk is bounded around the puncture (Theorem \ref{th:isosin}), and that any elliptic Weingarten peaked sphere with at most $2$ isolated singularities is a round sphere or a rotational football (Theorem \ref{2sing}).

Finally, in Section \ref{sec:yau} we solve a problem by Yau. Motivated by an example of Chern \cite{Ch0}, Yau asked in \cite{Yau} if any compact surface in $\R^3$ whose principal curvatures satisfy, in some order, the (hyperbolic) Weingarten equation $\kappa_1=c\kappa_2^3$ for some $c>0$, must be a rotational ellipsoid. We will construct a $C^2$ example that gives a negative answer to this question. The bifurcation phenomenon that allows the construction of this example is not possible in the elliptic case. The existence of this example was obtained by an adaptation of the phase space analysis of Section \ref{sec:fases}. However, we have written Section \ref{sec:yau} in an alternative way, so that it can be read independently from the rest of the paper.

The authors are grateful to José A. Gálvez for many valuable discussions during the preparation of this paper.

\section{The elliptic Weingarten equation}\label{sec:prelim}

There are several essentially equivalent ways of describing an elliptic Weingarten equation \eqref{eqw}. This is discussed in detail in the works \cite{GM4,FGM}. In this Section we explain how we will parametrize the Weingarten equation in a suitable way for our purposes.

Let $\Sigma$ be an oriented surface in $\R^3$ whose principal curvatures $\kappa_1\geq \kappa_2$ satisfy a Weingarten relation \eqref{eqw}, with $W\in C^1(\R^2)$. We will assume that \eqref{eqw} is elliptic, i.e., 
\begin{equation}\label{eli}
\frac{\parc W}{\parc k_1}\frac{\parc W}{\parc k_2}>0 \quad \mbox{on}\; \,  W^{-1}(0). 
\end{equation}

Each connected component of $W^{-1}(0)$ gives rise to a different elliptic theory, see \cite{FGM}. By \eqref{eli}, any such component of $W^{-1}(0)$ can be rewritten as  a proper curve in $\R^2$ given by a graph
\begin{equation}\label{eq:g}
\kappa_2=g(\kappa_1), \hspace{1cm} g'<0,
\end{equation}
where $g$ is $C^1$ in some interval of $\R$. Clearly, by the monotonicity of $g$ and the properness of the graph in \eqref{eq:g}, there exists a unique value $\alfa\in \R$ (which we call the \emph{umbilical constant} of the equation) such that $g(\alfa)=\alfa$. 

Recall that we are assuming $\kappa_1\geq \kappa_2$. Hence, the equation \eqref{eq:g} is only meaningful to us when restricted to that set. In this way, since $g$ is decreasing, we will assume that $g$ is defined in some interval $I_g\subset \R$ contained in $[\alfa,\8)$, with $\alfa\in I_g$. So, $g$ is a monotonic bijection from $I_g$ to $J_g:=g(I_g)\subset (-\8,\alfa]$. Thus, by the monotonicity and properness of $g$, there are two possibilities for the intervals $I_g$ and $J_g$; below, we denote $g(\8)={\rm lim} \, g(x)$ as $x\to \8$.

\begin{enumerate}
\item
$I_g=[\alfa,\8)$. In this case, $J_g=(b,\alfa]$, where $b=g(\8)\in (-\8,\alfa)\cup \{-\8\}$. 
\item
$I_g=[\alfa,b)$, $b\in\R$. In this case ${\rm lim}_{x\to b^-} \, g(x)=-\8$ and $J_g=(-\8,\alfa]$.
\end{enumerate}

The value of $\alfa$ has a geometrical meaning. If $\alfa=0$, planes are solutions to \eqref{eq:g}, while if $\alfa\neq 0$, spheres with both principal curvatures equal to $\alfa$ are solutions to \eqref{eq:g}. So, the geometric properties of these two classes of surfaces are different.

A change in the orientation of $\Sigma$ produces a surface satisfying a different elliptic Weingarten equation \eqref{eq:g}, according to the following correspondence:  
\begin{equation}\label{eq:cambiorienta}
(\kappa_1,\kappa_2,g(x),I_g,J_g,\alfa)\mapsto (-\kappa_2,-\kappa_1,-g^{-1}(-x), -J_g,-I_g,-\alfa).
\end{equation}
In particular, up to a change of orientation on the surface $\Sigma$, we can assume that the Weingarten equation \eqref{eq:g} on $\Sigma$ satisfies
\begin{equation}\label{eq:gorient}
I_g=[\alfa,\8). 
\end{equation}
In this respect, let us note that when $I_g=[\alpha,\8)$ and $J_g=(-\infty,\alpha]$, condition \eqref{eq:gorient} is preserved by a change of orientation. However, for $I_g=[\alpha,b)$ or $J_g=(b,\alpha]$, with $b\in\R$, the choice \eqref{eq:gorient} actually fixes an orientation on the surface. 	

We illustrate this formulation of elliptic Weingarten equations with some examples. The CGC equation $\kappa_1\kappa_2 =\alfa^2>0$ corresponds to the choice $g(x)=\alfa^2/x$, defined on $I_g=[\alfa,\8)$. For the CMC equation $\kappa_1+\kappa_2=2\alfa\in \R$, we have $g(x)=2\alfa-x$, defined on $I_g=[\alfa,\8)$. The non-symmetric linear elliptic Weingarten equation $a\kappa_1 + b\kappa_2 =0$ with $a,b>0$ corresponds to $g(x)= -\frac{a}{b}x$, defined on $I_g=[0,\8)$. 

As a more sophisticated example, consider as in \cite{FGM} the elliptic Weingarten equation $2H=K$, which can be rewritten as $W(\kappa_1,\kappa_2):=\kappa_1+\kappa_2 - \kappa_1 \kappa_2 =0$. The set $W^{-1}(0)$ has two connected components, that give rise to two different geometric theories, see \cite{FGM}. We choose the component that passes through the origin (thus, $\alfa=0$), which describes surfaces parallel to minimal surfaces. This component is given by $\kappa_2=g(\kappa_1)$, with $g(x)= \frac{x}{x-1}$, defined at first in $(-\8,1)$. However, the domain of $g$ when we only consider the graph $\kappa_2=g(\kappa_1)$ restricted to the half-plane $\kappa_1\geq \kappa_2$ is, in this case, $I_g=[0,1)$, since $g(0)=0$. Thus, $I_g$ is not of the form \eqref{eq:gorient}. In this situation, by a change of orientation for the equation as described in \eqref{eq:cambiorienta}, the function $g(x)$ transforms into $-g^{-1}(-x)$, that is, into $g(x)= -x/(x+1)$. For this new $g$, the restriction of the graph $\kappa_2=g(\kappa_1)$ to the half-plane $\kappa_1\geq \kappa_2$ is defined this time in $I_g=[0,\8)$, so now \eqref{eq:gorient} holds. It is with this form of the equation that we will work.

In summary, we can redefine the notion of elliptic Weingarten surface in an equivalent way, as follows:

\begin{definition}\label{def:w}
An \emph{elliptic Weingarten surface} is an oriented surface $\Sigma$ in $\R^3$ whose principal curvatures $\kappa_1\geq \kappa_2$ satisfy at every point the relation \eqref{eq:g} for some $g\in C^1(I_g)$, where $I_g=[\alfa,\8)$ and $g(\alfa)=\alfa$ for some $\alfa\in \R$. We also denote $b:=g(\8)\in[-\8,\alfa)$.

We let $\cW_g$ denote the class of all (oriented) elliptic Weingarten surfaces associated to a given function $g$ in these conditions.
\end{definition}

Even though it will not be used in this paper, we remark that a different way of writing the elliptic Weingarten equation is $H=F(H^2-K)$, where $H,K$ are the mean and Gauss curvatures, $F(t)$ is defined in $[0,\8)$ and satisfies the ellipticity condition $4t (F'(t))^2 <1$ for every $t>0$. For example, this is the formulation used by Sa Earp and Toubiana in \cite{SaTo,SaTo2}. The relation with \eqref{eq:g} is discussed in \cite{GM4,FGM}. 

\section{Rotational elliptic Weingarten surfaces: phase space analysis}\label{sec:fases}

Let $\Sigma$ be a rotationally invariant surface in $\cW_g$, see Definition \ref{def:w}, given as the rotation of a curve $\gamma(s)=(x(s),0,z(s))$ around the $z$-axis, where $s$ is the arc-length parameter of $\gamma(s)$ and $x(s)\geq 0$. Let $N$ denote the unit normal of $\Sigma$ with respect to which $\Sigma\in \cW_g$. By changing the orientation of the curve $\gamma$ if necessary, we may assume that $N$ is given along $\gamma(s)$ as $N(\gamma(s))= J \gamma'(s)$, where $J$ denotes the $\pi/2$-rotation in the $y=0$ plane. From here, the principal curvatures of $\Sigma$ are 
\begin{equation}\label{eq:princur}
\mu=\mu(s)=x'(s)z''(s)-x''(s)z'(s),\qquad  \landa=\landa(s) = \frac{z'(s)}{x(s)}.
\end{equation}
Since we are denoting $\kappa_1\geq \kappa_2$, we need to account for the two possibilities that, at each point $p\in \Sigma$, either $(\kappa_1,\kappa_2)=(\mu,\landa)$ or $(\kappa_1,\kappa_2)=(\landa,\mu)$.

As it was shown in \cite{SaTo}, due to ellipticity, either $\Sigma$ is  totally umbilical or it has no umbilical points. In particular, either $\mu\leq \landa$ and $\mu=g(\landa)$ on $\Sigma$, with $\landa\geq \alfa$,
or $\lambda\leq \mu$ and $\mu=g^{-1}(\landa)$ on $\Sigma$, with $\landa\leq \alfa$. 
Thus, $\Sigma$ satisfies the equation 
\begin{equation}\label{eq:mfl}
\mu=f(\landa),
\end{equation}
where $f:I_f  \to I_f$, with $I_f:=I_g\cup J_g$, is given by 
\begin{equation}\label{eq:f}
f|_{I_g} = g \qquad f|_{J_g} = g^{-1} . 
\end{equation}
Note that $I_f$ is an interval containing $\alfa$ in its interior, and $f$ is strictly decreasing and continuous, with $f\circ f= \mbox{Id}$. The function $f$ is $C^1$ in $I_f\setminus \{\alfa\}$ and, at  $\alfa$, $f$ has finite left and right derivatives. From \eqref{eq:gorient}, 
\begin{equation}\label{eq:forient}
I_f=(b,\8), 
\end{equation}
where $b:=g(\8)\in [-\8,\alfa)$.

Using \eqref{eq:princur} and $x'(s)^2+z'(s)^2=1$, we obtain from \eqref{eq:mfl} 
\begin{equation}\label{eq:x2}
z''(s) = f(\landa(s))\, x'(s), \hspace{1cm}
x''(s) = -x(s)\, \lambda(s)\, f(\lambda(s)), 
\end{equation}
and 
\begin{equation}\label{eq:edoraf}
 \landa' = \frac{x'}{x}(f(\landa)-\landa).
\end{equation}
In particular, $(x(s),\landa(s))$ is a solution to the following nonlinear autonomous system on any open interval where $x'(s)\neq 0$ and $x(s)> 0$:
\begin{equation}\label{eq:auto1}
\left\{\def\arraystretch{1.5}\begin{array}{lll} x' & = & \ep \sqrt{1-\landa^2 x^2}, \\ \landa' & = & \ep \sqrt{1-\landa^2 x^2} \, \frac{f(\landa)-\landa}{x}.\end{array} \right.
\end{equation}
Here $\ep ={\rm sign} (x') =\pm 1$. This process can be reversed, so that any solution to \eqref{eq:auto1} with $x(s)>0$ determines a rotational surface $\Sigma$ of the Weingarten class $\cW_g$. Thus, the orbits of \eqref{eq:auto1} will be identified with the profile curves of rotational surfaces in $\cW_g$, on open sets where $x'(s)\neq 0$ and $x(s)>0$.

The phase space of \eqref{eq:auto1} is the region (see Figure  \ref{fig:planodefases})
\begin{equation}\label{eq:R}
 \cR := \{(x,\landa) : x> 0, \; \landa>b, \; \landa^2 x^2 \leq 1\}.
\end{equation}
We will denote by $\Gamma$ the two boundary hyperbolas of $\cR$: 
\begin{equation}\label{eq:Gamma}
 \Gamma := \{(x,\landa) : x> 0,  \; \landa>b, \; \landa^2 x^2 = 1\}.
\end{equation}
Observe that $(x(s_0),\landa(s_0))$ lies in $\Gamma$ if and only if the generating curve $(x(s),z(s))$ of $\Sigma$ has at $s_0$ a point with vertical tangent vector.

\begin{center}\begin{figure}[htbp]
\includegraphics[width=.6\textwidth]{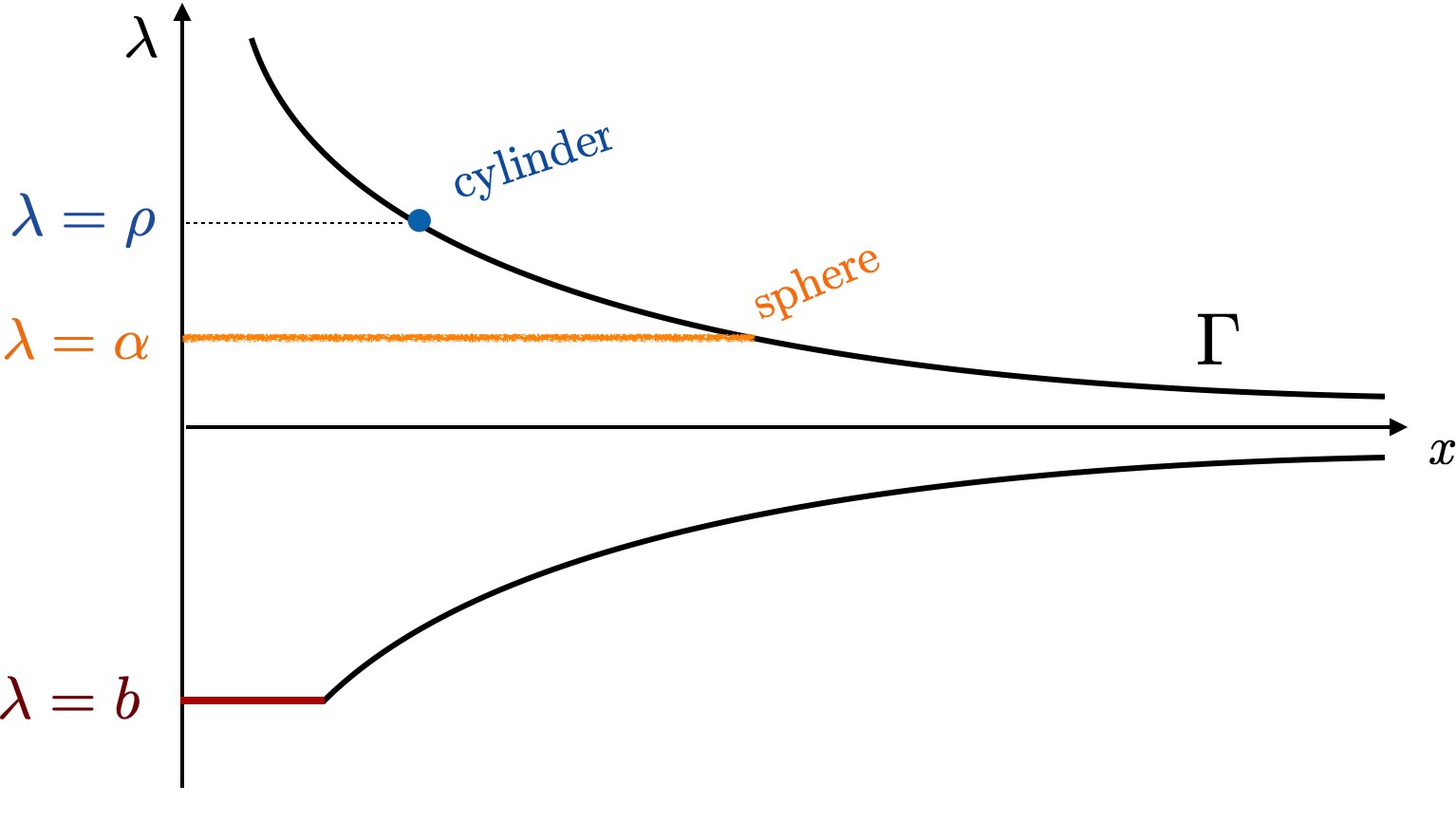}
\caption{The set $\cR$, its hyperbolic boundary $\Gamma$, and the orbits corresponding to spheres and cylinders, in the case $\alfa>0$, $b<0$.\label{fig:planodefases}}
\end{figure}\end{center}

Note that the systems \eqref{eq:auto1} for $\ep=1$ and $\ep=-1$ are equivalent, since they have the same orbits. Indeed, if  $(x(s),\landa(s))$ is a solution for $\ep=1$, then $(x(-s),\landa(-s))$ is a solution for $\ep =-1$. 

\begin{remark}\label{re:simetria}
The above property shows the following \emph{symmetry result} for rotational elliptic Weingarten surfaces. Assume that the generating curve $(x(s),z(s))$ of $\Sigma$ has a point $s_0$ with vertical tangent vector (and therefore $x^2\landa^2 =1$). Then, its associated orbit $(x(s),\landa(s))$ \emph{hits} $\Gamma$ at $s=s_0$, and then {\em bounces back}, following the same trajectory in the opposite sense, but with the sign of $\ep$ reversed. By this argument, we see that $(x(s),z(s))$ extends smoothly across $s_0$, with the sign of $x'(s)$ changing at $s_0$, and that $\Sigma$ is symmetric with respect to the horizontal plane passing through $(x(s_0),0,z(s_0))$.

In this way, the property that an orbit hits the boundary set $\Gamma$ does not produce any geometric complications (e.g. singularities) on $\Sigma$, and such an orbit can be considered to be regular up to $\Gamma$.
\end{remark}

There are two special orbits of this phase diagram that control the behavior of the rest of orbits of \eqref{eq:auto1}. See Figure \ref{fig:planodefases}. 

\begin{itemize}
\item
Let $\alfa\in\R$ be the umbilical constant of $\cW_g$, defined in Section \ref{sec:prelim} by the condition $g(\alfa)=\alfa$. Then, the segment (or half-line) $\cR\cap \{\lambda= \alfa\}$ defines an orbit of \eqref{eq:auto1}. If $\alfa=0$, it corresponds to a plane. If $\alfa\neq 0$, it corresponds to a sphere oriented so that its principal curvatures are given by $\alfa$. Orbits in $\cR$ that lie above $\cR\cap \{\lambda= \alfa\}$ correspond to surfaces for which $\kappa_1=\landa>\kappa_2=\mu$, whereas orbits below it give rise to surfaces with $\kappa_1=\mu>\kappa_2=\landa$. 
 \item 
Assume that there exists a cylinder $\Sigma$ in $\cW_g$. The necessary and sufficient condition for this is that $\alfa\neq 0$ and $b< 0$,  i.e., that $0\in I_f$, see \eqref{eq:forient}. The associated orbit of this cylinder is the point $(1/|\rho|,\rho)\in\Gamma$, where $\rho:=f(0)\in \R$. Since $|\alfa|<|\rho|$ by monotonicity of $f$, the point in $\cR$ corresponding to the cylinder is always in the connected component of $\cR\setminus\{\landa=\alfa\}$ not containing the $x$-axis.  \end{itemize}

Let $(x(s),\landa(s))$ be the  orbit for \eqref{eq:auto1}  associated to some rotational surface $\Sigma\in\cW_g$ different from the sphere or plane of $\cW_g$ and, if it exists, from the cylinder of $\cW_g$. Then, $(x(s),\landa(s))$ is contained in the region $\cR$ and, by solving \eqref{eq:edoraf}, we see that it can be reparametrized as a graph $(x(\landa),\landa)$, where 
\begin{equation}\label{eq:soledo}
x(\landa) =  x_0 \,\exp\int_{\landa_0}^{\landa} \frac{dt}{f(t)-t}, 
\end{equation}
for some $(x_0=x(\landa_0), \landa_0)\in \cR$, $\landa_0\neq \alfa$; here, one should recall that $\landa'(s)=0$ only happens at $\Gamma$, as discussed above. Conversely, given the function $x(\landa)$ as in \eqref{eq:soledo}, the restriction of the graph $(x(\landa),\landa)$  to $\cR$ describes an orbit of \eqref{eq:auto1}, and thus it gives rise to a rotational surface of the class $\cW_g$.

Lemma \ref{lem:xmonot} below describes the geometry of the curves $(x(\landa),\landa)$. See also Figure \ref{fig:reg_monot}.
\begin{center}\begin{figure}[htbp]
\includegraphics[width=\textwidth]{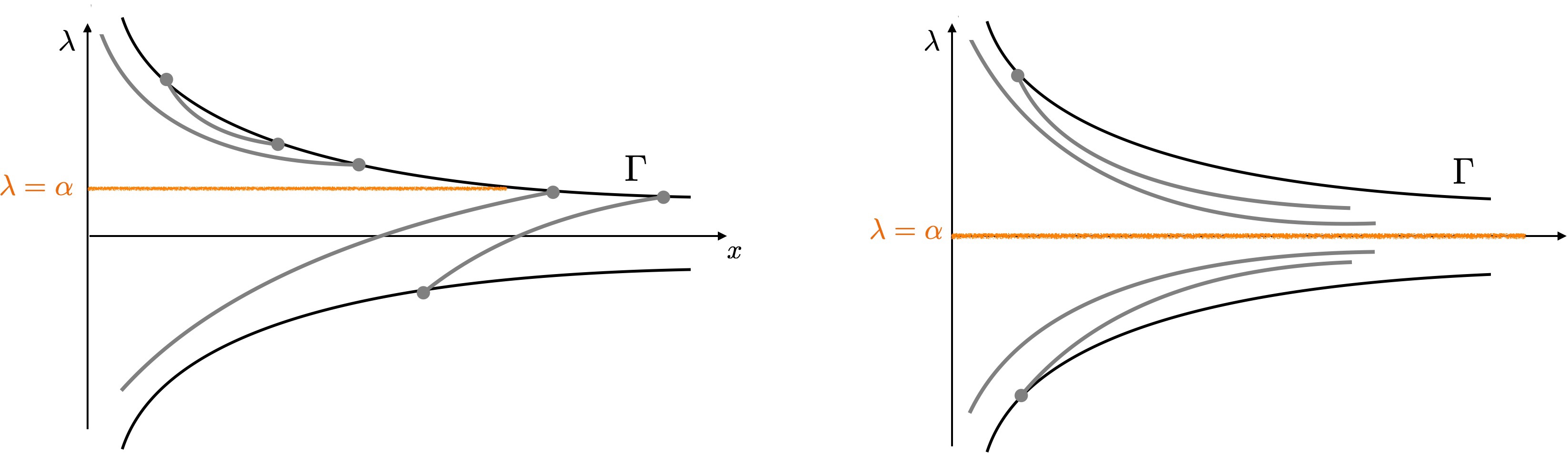}
\caption{Examples of possible behaviours of the orbits $(x(\landa),\landa)$. (Left: $\alfa>0$, right: $\alfa=0$).\label{fig:reg_monot}}
\end{figure}\end{center}
\begin{lemma}\label{lem:xmonot}
For any $(x_0,\landa_0)\in\cR$, $\landa_0\neq \alfa$, the graph $(x(\landa),\landa)$ of the function $x(\landa)$ given by \eqref{eq:soledo} is contained in $x>0$ and does not meet the line $\landa=\alfa$. 
Specifically,  if $\landa_0>\alfa$ (resp. $\landa_0<\alfa$) then  $x(\landa)$ is a decreasing  (resp. increasing) function defined in $(\alfa,\8)$ (resp. $(b,\alfa)$) with $x(\alfa)=\8$. Moreover:
\begin{enumerate}[i)]
\item Whenever $(x(\landa),\landa)$ intersects $\Gamma$, the intersection is transverse, except if $( x_0,\landa_0)=( 1/|\rho|, \rho).$
\item The restriction of the graph $(x(\landa),\landa)$ to $\cR$ is a connected regular curve, except for $(x_0,\landa_0) =( 1/|\rho|, \rho)$, where it consists uniquely of the point $\{(1/|\rho|,\rho)\}$. In particular, the graph $(x(\landa),\landa)$ intersects $\Gamma$ at two points, at most.
\end{enumerate}
\end{lemma}
\begin{proof} 
The monotonicity of $x(\landa)$ follows from the fact that $f$ is decreasing and $f(\alfa)=\alfa$. Besides,
since $f$ has finite left and right derivatives at $\alfa$, $$
\frac{f(\landa)-\landa}{\landa-\alfa}$$ is bounded between two negative constant around $\landa=\alfa$. Thus, the function $\landa\mapsto\int_{\landa_0}^\landa dt/(f(t)-t)$ diverges to $\8$ as $\landa$ approaches $\alfa$, what proves $x(\alfa)=\8$.

For $i)$, let $(x_0,\landa_0)\in \Gamma$, and view the component of $\Gamma$ that contains this point as the graph $x=\ep' /\landa$, where $\ep'\in \{0,1\}$ is the sign of $\landa_0$. A direct computation from \eqref{eq:soledo} with initial point $(x_0,\landa_0)$ shows that $x'(\landa_0)\neq -\ep' /\landa_0^2$ unless $\landa_0=\rho$ (recall that $\rho=f(0)$ and $f\circ f={\rm Id}$).
This proves $i)$.

For $ii)$, note that if $(x_0,\landa_0)\neq (1/|\rho|,\rho)$, there exists $\landa_1$ close to $\landa_0$ such that $|\landa_1\, x(\landa_1)|<1$. Also, as  $x(\alfa)=\8$,  for $\landa$ close enough to $\alfa$ we have $|\landa\, x(\landa)|>1$; thus, $(x(\landa),\landa)$ intersects $\Gamma$. If the restriction to $\cR$ of $(x(\landa),\landa)$  had more than one connected component, there would be at least two different points at which the derivative of the function $\landa\, x(\landa)$ vanishes. But 
\begin{equation}\label{monang}
\frac{d}{d\landa} (\landa\, x(\landa) )= x(\landa) \frac{f(\landa)}{f(\landa)-\landa},
\end{equation}
which vanishes only at $\landa=\rho$. So, the restriction of $(x(\landa),\landa)$ to $\cR$ is a connected curve. As a consequence of the above expression, we also deduce that $|\landa \, x(\landa)|$ attains a minimum at $\landa=\rho$, proving that for $(x_0,\landa_0) = (1/|\rho|,\rho)$ the graph of $x(\landa)$ stays {\em outside} $\cR$ except for that point. 
\end{proof}

\begin{remark}\label{re:extreme}
As a consequence of Lemma \ref{lem:xmonot}, any point $(x_0,\landa_0)\in \Gamma$, with $\landa_0\neq \rho$ if $b<0$, is the endpoint of exactly one orbit of \eqref{eq:auto1} in $\cR$, and this orbit intersects $\Gamma$ transversely.  
\end{remark}

\section{Singularities of rotational elliptic Weingarten surfaces}\label{sec:existence}

\subsection{Types of singularities}
Let $\Sigma$ be a rotational surface in $\cW_g$  given as the rotation of an arc-length parameterized curve $\gamma(s)=(x(s),0,z(s))$ around the $z$-axis. We will assume that $\gamma(s)$ is defined in some interval $\mathcal{J}\subset \R$, but that $\Sigma$ cannot be extended to a regular surface as $s\to a$, where $a\in \R$ is one of the endpoints of $\mathcal{J}$. That is, the surface $\Sigma$ becomes \emph{singular} as $s\to a$.

Due to our previous study of the phase plane of \eqref{eq:auto1}, the orbit associated to $\Sigma$ must approach the boundary of $\cR$ as $s\to a$. More specifically, as we know that $\gamma(s)$ can be smoothly extended if $(x(s),\landa(s)) \to p_0\in \Gamma$ as $s\to a$ (see Remark \ref{re:simetria}), there are only two possible ways in which a singularity can happen: either $x(s)\to 0$ as $s\to a$, or $b\neq -\8$ and $\landa(s)\to b$ as $s\to a$. The first case corresponds to an isolated singularity created as the surface touches its rotational axis. In the second one, the surface is singular along a circle, since $x(b)>0$ in \eqref{eq:soledo}. In both situations, the surface is a graph near the singularity.

\subsection{The function $\cG(\landa)$} We next define a function that will control the existence of singularities of elliptic Weingarten surfaces. In the proof of Lemma \ref{lem:xmonot}, we showed that the behavior of  $\landa \,  x(\landa)$ plays an important role in the study of the phase space of \eqref{eq:auto1}. In this spirit, we consider the function $\cG: (b,\8)\setminus\{\alfa\} \to\R$ given by 
\begin{equation}\label{eq:G}
\cG(\landa) = \landa\,  \exp \cF(\landa) ,
\end{equation}
where $\cF(\landa)$ is any primitive of $1/(f(\landa)-\landa)$. Note that $\cG(\landa)$ is defined up to a positive multiplicative constant on each of the two connected components of $(b,\8)\setminus\{\alfa\}$.

\begin{remark}\label{re:G}
If $x(\landa)$ is given by \eqref{eq:soledo} for some $(x_0,\landa_0)$, with $\landa_0\neq \alfa$, the function  $\landa\,   x(\landa)$ and the restriction of $\cG(\landa)$ to the interval of $(b,\8)\setminus\{\alfa\}$ that contains $\landa_0$ differ by a multiplicative positive constant: 
\begin{equation}\label{eq:Gx}
\frac{\landa\,   x(\landa)}{\landa_0\,   x(\landa_0)} = \frac{\cG(\landa)}{\cG(\landa_0)}
\end{equation}
In particular, they have the same monotonicity. The same holds for $|\landa\,   x(\landa)|$ and $|\cG(\landa)|$. 
\end{remark}
By a direct computation,
\begin{equation}\label{deg}
\cG'(\lambda)=\frac{f(\lambda)}{f(\lambda)-\lambda} \exp \cF(\landa).
\end{equation}
This equation allows to study the monotonicity regions of $\cG$, as described in the following lemma. See also Figure \ref{fig:G}.

\begin{lemma}\label{lem:Gmonot}
The function $\cG: (b,\8)\setminus\{\alfa\}\to \R$ defined in \eqref{eq:G} satisfies: 

\begin{enumerate}
\item If $b\geq 0$, then $\cG(\landa)$ is increasing for $\landa \in (b,\alfa)$ and decreasing in $(\alfa,\8)$. 
\item If $b<0$ and $\alfa\neq 0$, then $\cG(\landa)$ is decreasing for $\landa$ between $\alfa$ and $\rho=f(0)$, and increasing otherwise. Also, $\cG(\alfa)=\pm\8$, where $\pm = \rm{sign}(\alfa) $.
\item If $\alfa =0$, $\cG$ is increasing on $I_f$, and $\cG(0)=0$. 
\end{enumerate}
In particular, $\cG(b)=\lim_{\landa\to b}\cG(\landa)$ and $\cG(\8)=\lim_{\landa\to\8}\cG(\landa)$ exist (although they could be infinite). If $b\neq -\8$,  $\cG(b)$ is  finite. 
\end{lemma}

\begin{proof}
All the monotonicity properties in the statement, as well as the existence of the limits of $\cG$ at $\landa=\8$ and $\landa=b$, follow directly from \eqref{deg} and a case by case study, using that $f$ is decreasing, and $f(\alfa)=\alfa$. It remains to analyze $\cG$ around $\alfa$.

If $\alfa\neq 0$, the behaviour of $\cG$ at $\alfa$ follows from \eqref{eq:Gx} and the fact that $ x(\alfa)=\8$ (Lemma \ref{lem:xmonot}). If $\alfa=0$, it is easy to check that $\cG(\landa)<0$ and $\cG'(\landa)>0$ for $\landa<0$, whereas $\cG(\landa)>0$ and $\cG'(\landa)>0$ for $\landa>0$, what ensures that the lateral limits of $\cG(\landa)$ at $\landa=0$ exist and are finite. Moreover, as $f(\landa)$ has finite, negative derivatives at $\landa=0$ for $\landa\to 0^+$ and $\landa\to 0^-$, we have 
$$ c_1\landa\leq f(\landa) \leq c_2\landa $$
near the origin, for suitable constants $c_1,c_2 <0$, from where we obtain from \eqref{eq:G}
$$A_1 \landa^{b_1} \leq \cG(\landa) \leq A_2\landa^{b_2},$$
where $A_i,b_i>0$, $i=1,2$. In particular,  $\cG(0^+)=0=\cG(0^-)$.

Finally, if $b\neq -\8$, the function $1/(f(\landa)-\landa))$ extends continuously to $\landa=b$ (with value $0$), and so $\cG(b)$ is finite.
\end{proof}

\begin{center}\begin{figure}[hbtp]
\includegraphics[width=.9\textwidth]{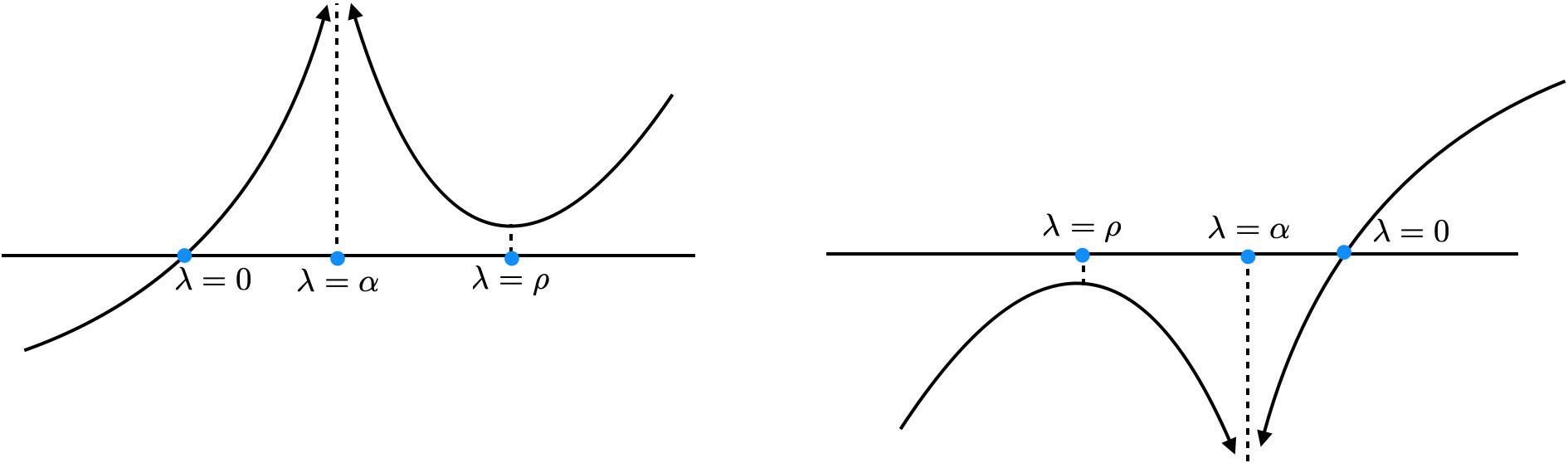}
\caption{Monotonocity of $\cG(\landa)$ when $b<0$. 
 Left: when $\alfa>0$. Right: when $\alfa<0$. }\label{fig:G}
\end{figure}\end{center} 

\subsection{A characterization of existence of singularities} By our discussion in Section \ref{sec:fases}, we know that the non-horizontal orbits of \eqref{eq:auto1} are the restriction to $\cR$ of graphs $(x(\landa),\landa)$ given by \eqref{eq:soledo}, and that the horizontal orbit $\{\landa=\alfa\}\cap \cR$ corresponds to the totally umbilic example in the class $\cW_g$, which does not have singularities. This justifies the following definition.

\begin{definition}\label{def:rsin}
Let $\cR_{\rm sing}$ be the set of points $(x_0,\landa_0)\in \cR$ such that the orbit of \eqref{eq:auto1} that passes through $(x_0,\landa_0)$ corresponds to an elliptic Weingarten surface that has some singularity. For any such $(x_0,\landa_0)\in \cR_{\rm sing}$, we necessarily have $\landa_0\neq \alfa$, and so we can decompose 
$$\cR_{\rm sing}= \cR_{\rm sing}^+\cup \cR_{\rm sing}^- , \hspace{0.5cm} \cR_{\rm sing}^+ := \cR_{\rm sing}\cap \{\landa>\alfa\}, \hspace{0.2cm} \cR_{\rm sing}^- := \cR_{\rm sing}\cap \{\landa<\alfa\}.$$
\end{definition}

By this definition, a class $\cW_g$ of elliptic Weingarten surfaces has rotational examples with singularities if and only if $\cR_{\rm sing}$ is non-empty. We next characterize this condition. In the theorem below, $b$ is defined as usual by $b:=g(\8)\in [-\8,\alfa)$, and $\cG:(b,\8)\setminus \{\alfa\}\flecha \R$ is the function introduced in \eqref{eq:G}. Also, define $\cR^+:=\cR\cup \{\landa>\alfa\}$ and $\cR^- :=\cR\cup \{\landa<\alfa\}$.

\begin{theorem}\label{th:existence}
Let $\cW_g$ be a class of elliptic Weingarten surfaces. Then:
\begin{enumerate}
\item
The set $\cR_{\rm sing}^+$ is non-empty if and only if $|\cG(\8)|<\8$, and in that case, $$\cR_{\rm sing}^+ =\left\{(x_0,\landa_0)\in \cR^+ : |\landa_0 x_0|\leq \frac{|\cG(\landa_0)|}{|\cG(\8)|}\right\}.$$
\item
The set $\cR_{\rm sing}^-$ is non-empty if and only if $|\cG(b)|<\8$, and in that case, $$\cR_{\rm sing}^- =\left\{(x_0,\landa_0)\in \cR^- : |\landa_0 x_0|\leq \frac{|\cG(\landa_0)|}{|\cG(b)|}\right\}.$$
\end{enumerate}
\end{theorem}
\begin{proof}
As we already discussed at the start of Section \ref{sec:existence}, any orbit generates an elliptic Weingarten surface with singularities if and only if it approaches $x=0$ or $\landa =b$ in $\cR$. Thus, since $b<\alfa$, it follows from item ii) in Lemma \ref{lem:xmonot} that the orbits $(x(\landa),\landa)$ in $\cR_{\rm sing}^+$ are exactly those for which $\landa>\alfa$ and $|\landa x(\landa)|<1$ for all $\landa>\landa_0$, where $(x_0,\landa_0)\in \cR^+$ is a point of the orbit. In a similar way, the orbits in $\cR_{\rm sing}^-$
are those that satisfy $|\landa x(\landa)|<1$ for all $\landa \in (b,\landa_0)$, where $(x_0,\landa_0)\in \cR^-$. To see this one should note that, by \eqref{eq:soledo}, $x(\landa)\to 0$ happens only if $\landa \to \pm \8$.

In addition, it follows from Remark \ref{re:G} that, for an orbit $(x(\landa),\landa)$ passing through $(x_0,\landa_0)$, the condition $|\landa x(\landa)|<1$ holds for some $\landa$ if and only if $$|\cG(\landa)|<\frac{|\cG(\landa_0)|}{|\landa_0 x_0|}.$$

We now take limits in this expression as $\landa\to \8$ or $\landa\to b$. By the discussion above and the monotonicity properties of $\cG$ at $\landa=\8$ and $\landa=b$ described in Lemma \ref{lem:Gmonot}, we deduce that if $(x_0,\landa_0)\in \cR_{\rm sing}^+$ (resp. $(x_0,\landa_0)\in \cR_{\rm sing}^-$), then $\cG(\8)$ (resp. $\cG(b)$) has a finite value, and  
\begin{equation}\label{fineq}
 |\landa_0 x_0| \leq \frac{|\cG(\landa_0)|}{|\cG(\8)|} \hspace{0.5cm} \left(\text{resp.} \  |\landa_0 x_0| \leq \frac{|\cG(\landa_0)|}{|\cG(b)|}\right).
 \end{equation}

We next prove the converse inclusion for the first assertion of Theorem \ref{th:existence}; the same proof works for the second assertion, substituting $\8$ by $b$. Thus, assume that $|\cG(\8)|<\8$, and let $(x(\landa),\landa)$ be an orbit that contains a point $(x_0,\landa_0)\in \cR^+$ that satisfies the first inequality in \eqref{fineq}. By Lemma \ref{lem:Gmonot}, for $\landa$ close to $\8$ we have $|\cG(\landa)|<|\cG(\8)|$. This, together with \eqref{eq:Gx} and \eqref{fineq}, give that 
$$ |\landa x(\landa) | = |\landa_0 x_0| \frac{|\cG(\landa)|}{|\cG(\landa_0)|}\leq \frac{|\cG(\landa)|}{|\cG(\8)|}<1,$$ that is, $(x_0,\landa_0)\in \cR_{\rm sing}^+$. This completes the proof of Theorem \ref{th:existence}.
\end{proof}

\begin{corollary}
Let $\cW_g$ be a class of elliptic Weingarten surfaces. Then, a necessary and sufficient condition for the existence of rotational examples of $\cW_g$ with singularities is that at least one of the following two conditions hold:
\begin{enumerate}
\item
$|\cG(\8)|<\8$.
 \item
$|\cG(b)|<\8$.
\end{enumerate}
If $b\neq -\8$, the second condition always holds, by Lemma \ref{lem:Gmonot}.
\end{corollary}

\subsection{Isolated singularities} The following lemma describes the local behavior of a radial elliptic Weingarten graph around an isolated singularity.

\begin{lemma}[{\bf Isolated singularities}]\label{lem:isi}
Let $\Sigma$ be a radial elliptic Weingarten graph $z=u(r)$, $r:=\sqrt{x^2+y^2}$, defined in a punctured disk around the origin. Assume that $\Sigma$ is not totally umbilic. Then, $\Sigma$ has at the origin a non-removable isolated singularity of conelike type, that is:
\begin{enumerate}
\item
The principal curvature $\landa(r)\to \pm \8$ as $r\to 0$.
\item
$u'(r)$ is monotonic for small values of $r$, and thus has a limit, maybe infinite, as $r\to 0$.
 \item
$u(r)$ is monotonic for small values of $r$, and has a finite limit as $r\to 0$. 
\end{enumerate}
\end{lemma}
\begin{proof}
Let $\cW_g$ be the elliptic Weingarten class in which $\Sigma$ lies. Let $(x(\landa),\landa)$, with $x(\landa)$ as in \eqref{eq:soledo}, be the trajectory in $\cR$ of the orbit $(x(s),\landa(s))$ of \eqref{eq:auto1} associated to $\Sigma$, as already discussed; note that, since $\Sigma$ is not totally umbilic, the orbit is certainly of this form. By \eqref{eq:soledo}, $x(\landa)\to 0$ happens only if $\landa \to \pm \8$. This proves the first assertion. Observe that $\landa \to -\8$ is only possible if $b=-\8$, since otherwise $(x(\landa),\landa)$ does not stay in $\cR$ as $\landa\to -\8$.

The second assertion is a direct consequence of the general fact that $u'(r)$ is monotonic on any rotational elliptic Weingarten graph $z=u(r)$ in $\R^3$; see e.g. Lemma 4.4 in \cite{GM4}, where this result is proved in a more general context.

To prove the third assertion, we only need to rule out the case that $u(r)\to \pm \8$.  In that case, $|u'(r)|\to \8$. Let us parametrize the profile curve $(r,0,u(r))$ by arc-length as $(x(s),0,z(s))$, with $s\to \8$ as $r\to 0$. By $|u'(r)|\to \8$ as $r\to 0$, we have $x'(s)\to 0$ as $s\to \8$. Also, by \eqref{eq:princur} and the second equation in \eqref{eq:x2}, we deduce that $x''(s)=-z'(s)f(\landa(s))$. From this equation, using the monotonicity of $f$ and $|\landa(s)|\to \8$, we obtain from $z'(s)\to \pm 1$ that $x''(s)$ has a limit at $s\to \8$. Since $x'(s)\to 0$, this limit must be $0$. Thus, $f(\landa(s))\to 0$ as $s\to \8$. Since $|\landa(s)|\to \8$ as $s\to \8$, this is only possible if $\landa(s)\to \8$ and $f(\8)=b=0$. This implies that both principal curvatures of $\Sigma$ must have the same sign around the isolated singularity, what is an obvious contradiction with the fact that $u(r)\to \pm \8$ as $r\to 0$.
\end{proof}

Let $N$ denote the unit normal of the graph $\Sigma$ of Lemma \ref{lem:isi}, and write $e_3=(0,0,1)$. By the monotonicity of $u'(r)$ proved in Lemma \ref{lem:isi}, there exists a \emph{limit angle} $\nu_0\in [-1,1]$ at the isolated singularity of $\Sigma$, defined as 
\begin{equation}\label{eq:an}
\nu_0 :=\lim_{r\to 0^+} \esiz N,e_3\esde=\lim_{r\to 0^+}\frac{\ep}{\sqrt{1+u'(r)^2}},
\end{equation} 
where $\ep=1$ if $N$ points upwards, and $\ep =-1$ otherwise. If we use the parametrization of $\Sigma$ in terms of the profile curve $\gamma(s)=(x(s),0,z(s))$ as in the beginning of Section \ref{sec:fases},  and assume $x(s)\to 0$ as $s\to a$, then we see that $\ep ={\rm sign} (x'(s))$ in \eqref{eq:an}, and that, by \eqref{eq:princur},
\begin{equation}\label{foran}
\nu_0 = \lim_{s\to a} \ep \, \sqrt{1- \landa(s)^2 x(s)^2}.
\end{equation}

We next describe this limit angle $\nu_0$ in the phase space $\cR$, and show that it can be used to parametrize the space of rotational elliptic Weingarten surfaces with an isolated singularity.

\begin{proposition}\label{propang}
Let $\Sigma$ be a rotational surface of the elliptic Weingarten class $\cW_g$, and assume that $\Sigma$ has an isolated singularity. Let $(x(\landa),\landa)$ be the orbit in the phase space $\cR$ associated to $\Sigma$, and take any point $(x_0,\landa_0)$ of that orbit. 

Then, if $(x_0,\landa_0)\in \cR_{\rm sing}^+$ (resp. if $(x_0,\landa_0)\in \cR_{\rm sing}^-$, and so in particular $b=-\8$), the limit angle $\nu_0$ of $\Sigma$ at the singularity, defined by \eqref{eq:an}, satisfies $\nu_0\neq 0$ and is given by 
 \begin{equation}\label{singang}
 \nu_0 = \frac{\ep}{|\cG(\landa_0)|} \sqrt{\cG(\landa_0)^2-(\landa_0\,  x_0 \, \cG(\pm \8))^2}, 
  \end{equation}
where the sign of $\pm \8$ corresponds to the one in $\cR_{\rm sing}^{\pm}$.

Moreover, in the above conditions, two different orbits in $ \cR_{\rm sing}^+$ (resp. in $ \cR_{\rm sing}^-$) always have different limit angles at the isolated singularity.
\end{proposition}
\begin{proof}
We prove the result for $(x_0,\landa_0)\in \cR_{\rm sing}^+$; the case $(x_0,\landa_0)\in \cR_{\rm sing}^-$ is similar. By the monotonicity of $|\landa x(\landa)|$ as $\landa \to \8$, we deduce from \eqref{eq:Gx} and \eqref{foran} that
 \begin{equation}\label{foran2}
\sqrt{1-\nu_0^2}= \lim_{\landa\to \8} |\landa x(\landa) |= \frac{|\landa_0\,  x_0 \, \cG(\8)|}{|\cG(\landa_0)|}\in (0,1].
 \end{equation} 
This yields \eqref{singang}.

To prove the last assertion, consider two orbits $(x_1(\landa),\landa)$, $(x_2(\landa),\landa)$ contained in $\cR_{\rm sing}^+$ in the above conditions. Assume that the limit angles at the isolated singularity of both orbits coincide. By \eqref{foran2}, we obtain for every $\landa_0>\alfa$ that $$\left| \frac{\landa_0 x_1(\landa_0) \cG(\8)}{\cG(\landa_0)}\right|=\left| \frac{\landa_0 x_2(\landa_0) \cG(\8)}{\cG(\landa_0)}\right|.$$ Hence, $x_1(\landa_0)=x_2(\landa_0)$, for all $\landa_0>\alfa$. This completes the proof.
\end{proof}

A Weingarten class $\cW_g$ given by \eqref{eq:g} is said to be {\em uniformly elliptic} if there exist constants $\Lambda_1,\Lambda_2<0$ such that 
\begin{equation}\label{eq:unif}
\Lambda_1 < g'(t) < \Lambda_2, \quad \forall t\in I_g.
\end{equation} 
In particular, $b=g(\8)=-\8$. See e.g. \cite{FGM}. Theorem \ref{th:existence} can be used to prove that rotational, uniformly elliptic Weingarten surfaces do not present singularities:

\begin{corollary}\label{co:uniform}
Let  $\Sigma$ be a rotational uniformly elliptic Weingarten surface. Then,  $\Sigma$  cannot have singularities, and is a piece of one of the complete examples described by Sa Earp and Toubiana in \cite{SaTo,SaTo2}. 
\end{corollary}

\begin{proof}
According to Theorem \ref{th:existence}, it is enough to prove that $|\cG(\8)|=|\cG(-\8)|=\8$. First, recall from \eqref{eq:f} that $g(t)=f(t)$ for all $t>\alfa$. Thus, 
since $g$ satisfies \eqref{eq:unif} and $f(\alfa)=\alfa$, we have for all $t>\alfa$ 
$$M_1(t-\alfa)\leq f(t)-t\leq M_2(t-\alfa)<0,$$ where $M_i:=\Lambda_i -1<-1$, $i=1,2$. From there,
$$C_2\lambda (\lambda-\alfa)^{1/M_2} \leq \cG(\landa) \leq C_1\lambda (\lambda-\alfa)^{1/M_1}$$
for some $C_1,C_2>0$ and for all $\landa>\max\{\alfa,0\}$. Thus, $|\cG(\8)|=\8$. 

An analogous computation using that $f(t)=g^{-1}(t)$ for all $t\in (b,\alfa)$ shows that $|\cG(b)|=|\cG(-\8)|=\8$, what finishes the proof.  
\end{proof}


\section{Rotational elliptic Weingarten surfaces: classification}\label{sec:classification}
In this section we give the classification of all rotational elliptic Weingarten surfaces, using the theoretical frame of Sections \ref{sec:prelim} and \ref{sec:fases} and the study of singularities of Section \ref{sec:existence}.

Along this section, we let $\cW_g$ denote some class of elliptic Weingarten surfaces (Definition \ref{def:w}). Associated to the function $g$ that defines $\cW_g$, we recall the definition of the function $f$ in \eqref{eq:f}. 
For classification purposes, we will divide elliptic Weingarten classes $\cW_g$ into three natural subclasses, that present quite different geometric behaviors, and proceed to classify rotational surfaces within each subclass separately.

\subsection{Weingarten surfaces of CGC type ($b\geq 0$)}
If $b\geq 0$ for an elliptic Weingarten class $\cW_g$, any surface in $\cW_g$ has positive curvature, and in particular there are no cylinders (or planes) in $\cW_g$. The constant positive curvature equation $K=c>0$ is a particular example of this subclass of elliptic Weingarten equations.

In this $b\geq 0$ situation, any orbit $(x(\landa),\landa)$ of \eqref{eq:auto1} in $\cR$ satisfies that $|\landa x(\landa)|=\landa x(\landa)$ is strictly decreasing in  $(\alpha,\8)$ (see Lemma \ref{lem:Gmonot} and Remark \ref{re:G}). Then, for every $\landa>\landa_0>\alfa$ we have $$|\landa x(\landa)| < |\landa_0 x(\landa_0)| \leq 1.$$ This shows that $\cR_{\rm sing}^+ = \cR\cap \{\landa>\alfa\}$, see Definition \ref{def:rsin}. A similar argument replacing $\8$ by $b\geq 0$ gives $\cR_{\rm sing}^- = \cR\cap \{\landa<\alfa\}$. In this way, any rotational surface in $\cW_g$ that is not totally umbilic has singularities (isolated if $\landa>\alfa$ and singular curves if $\landa<\alfa$).  Moreover, since $x(\alfa) =\8$ (Lemma \ref{lem:xmonot}) and $b\geq 0$, the orbit meets $\Gamma$ at a unique point, see Figure \ref{fig:Kpositiva}. 

\begin{center}\begin{figure}[htbp]
\includegraphics[width=.75\textwidth]{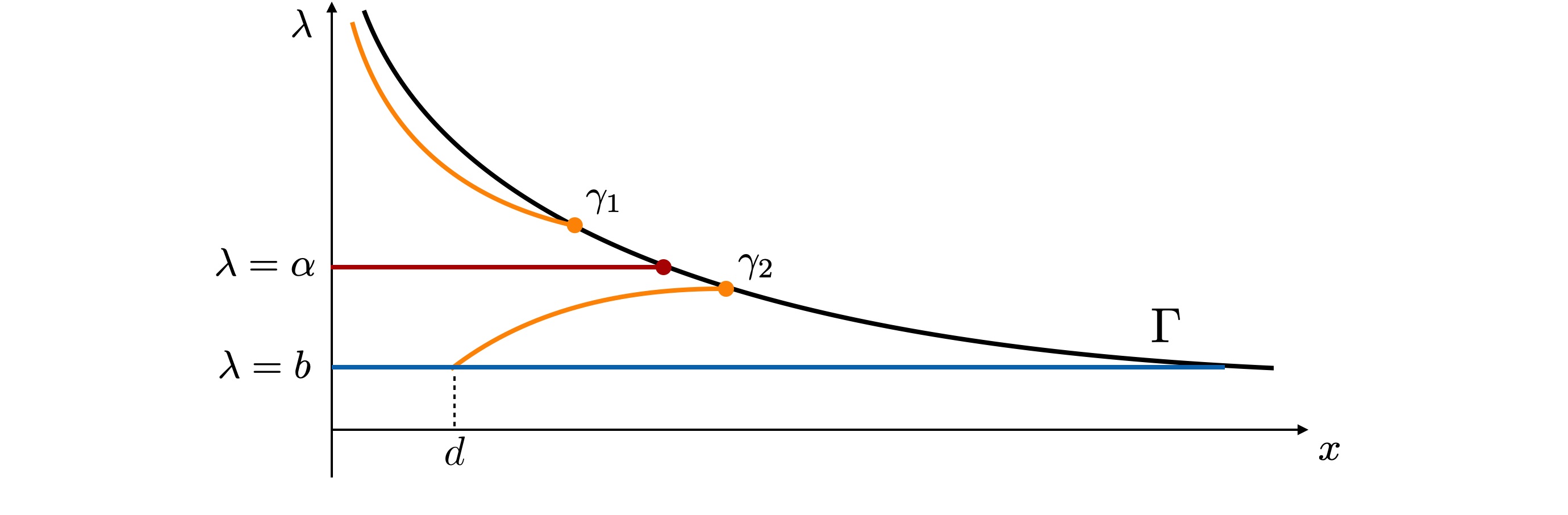}
\caption{Phase space when $b\geq 0$. The orbit $\gamma_1$ correspond to a surface of football type, and $\gamma_2$ describes a bracelet with singular curves at a distance $x_0>0$ of the rotational axis. }\label{fig:Kpositiva}
\end{figure}\end{center}

Thus, the corresponding rotational surfaces in $\cW_g$ for $\landa>\alfa$ are convex topological spheres with two singularities and a horizontal plane of symmetry (see Remark \ref{re:simetria}). We will call these surfaces {\bf (American) footballs}.  On the other hand, rotational surfaces with $\landa<\alfa$ are convex annuli with a horizontal plane of symmetry,  bounded by two singular curves. We will call these surfaces {\bf bracelets}. See Figure \ref{fig:esfera_sing}.

\begin{center}\begin{figure}[htbp]
\includegraphics[width=.65\textwidth]{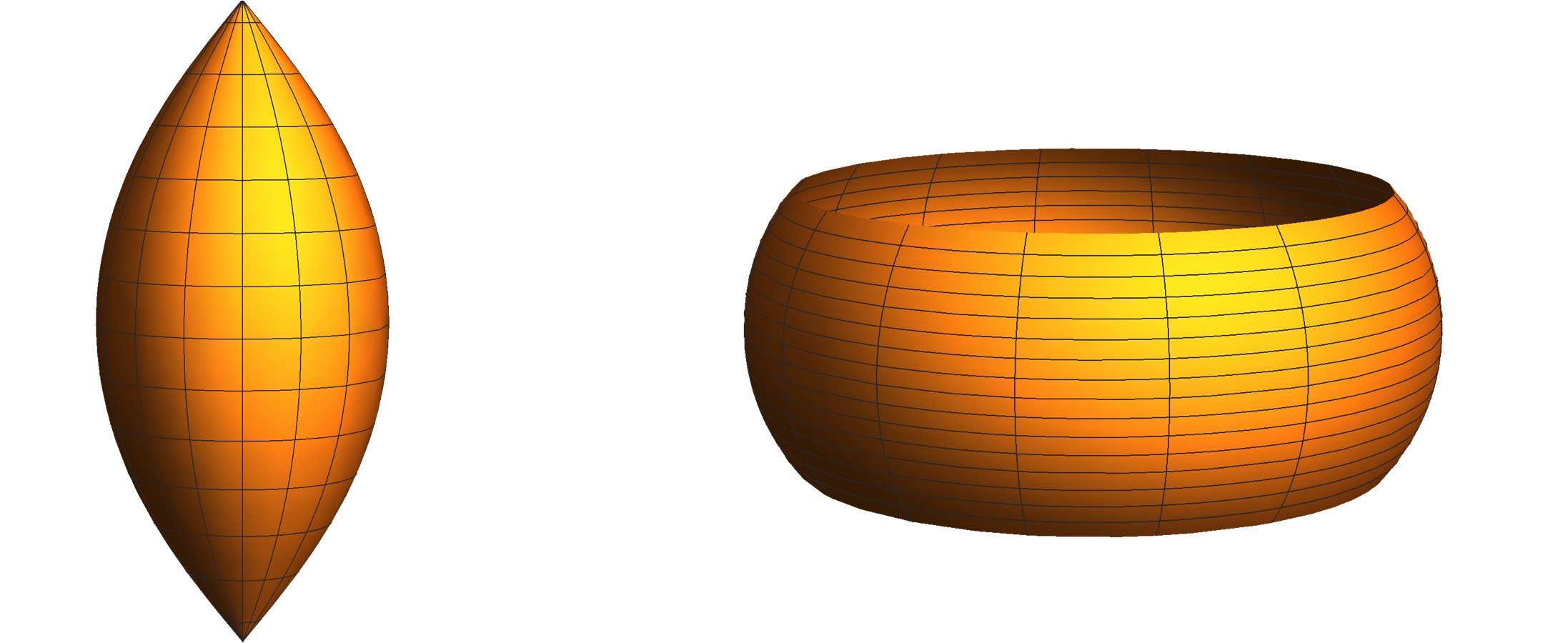}
\caption{(American) footballs and bracelets}\label{fig:esfera_sing}
\end{figure}\end{center}

As seen in Proposition \ref{propang}, the class of footballs can be parametrized in terms of the limit angle $\nu_0$ at their two singularities (by symmetry, if the limit angle at one singularity is $\nu_0>0$, at the other one is $-\nu_0$, on any given football). In this case, taking into account \eqref{singang} and choosing as initial condition the point $(x_0,\landa_0)$ at which the orbit meets $\Gamma$, it is easy to see that $\nu_0$ can take any value in $(0,1)$. We denote by $\cF_\nu$ the football of the class $\cW_g$ with limit angles $\{\nu,-\nu\}$ at its two singularities. Similarly, we will let $\cB_{x_0}$ denote the bracelet of the class $\cW_g$ whose singular curves are at a distance $x_0$ from the rotation axis. It is clear by the phase portrait in Figure \ref{fig:Kpositiva} that such an example exists and is unique, for each $x_0\in (0,1/b)$. 

This leads to the following classification result: 
\begin{theorem}\label{th:Kpositiva}
Let $\cW_g$ be a class of elliptic Weingarten surfaces with $b\geq 0$. Then any rotational surface in $\cW_g$ is an open piece of:
\begin{enumerate}
\item
A sphere of principal curvatures $\alfa$,
 \item
a bracelet $\cB_{x_0}$ for some $x_0\in (0,1/b)$, or
 \item
a football $\cF_\nu$ for some $\nu\in (0,1)$.
\end{enumerate}
\end{theorem}

\subsection{Weingarten surfaces of minimal type ($\alfa=0$)}\label{sec:a0}

If $\alfa=0$ holds for an elliptic Weingarten class $\cW_g$, then every plane of $\R^3$ is an element of $\cW_g$, and in particular there are no compact surfaces (without boundary) in $\cW_g$. The simplest case of such elliptic Weingarten class is given by minimal surfaces $(H=0)$.

Fix an elliptic Weingarten class $\cW_g$ with $\alfa=0$. Take any $(x_0,\landa_0)\in\Gamma$, and let $(x(\landa),\landa)$ be the orbit of \eqref{eq:auto1} that contains $(x_0,\landa_0)$. By \eqref{monang}, the function $\landa\mapsto \landa x(\landa)$ is strictly increasing. Thus, the orbit $(x(\landa),\landa)$ never touches again $\Gamma$. In this way, $x(\landa)$ is defined for every $\landa$ between $0$ and $\landa_0\in \R$, and $x(\landa)\to \8$ as $\landa\to 0$. 

These orbits describe the {\bf special catenoids} of the Weingarten class $\cW_g$, classified in \cite{SaTo2}. They are complete embedded surfaces, given as the rotation around the $z$-axis of a convex planar curve $(x(z),0,z)$, with $z\in (-z_1,z_1)$, where $z_1\in\R^+\cup\{\8\}$. The function $x(z)$ is positive and even, with $x(z)\to \8$ as $z\to \pm z_1$, and it attains its minimum at $\tau=x(0)$, which is called the {\em necksize} of the special catenoid. Depending on our choice of $g$, the value of $z_1$ can be infinite or finite, i.e. the special catenoid can have either finite or infinite height; see Theorem \ref{heica}. The principal curvature $\landa$ never vanishes on a special catenoid, and one can observe that if $\landa>0$, its unit normal $N$ is the interior one, while if $\landa<0$, it is the exterior one. For each $\tau>0$, we denote by $C^+_{\tau}$ (resp. $C^-_{\tau}$) the special catenoid with necksize $\tau$ and an interior (resp. exterior) unit normal. For the case in which $\cW_g$ is the minimal surface class, i.e., when $g(x)=-x$, then $C^+_{\tau}$ and $C^-_{\tau}$ are the same surface (the usual minimal catenoid of necksize $\tau$) with opposite orientations. However, in the case of a general elliptic Weingarten class $\cW_g$ with $\alfa=0$, the special catenoids $C^+_{\tau}$ and $C^-_{\tau}$ have different geometric behaviors.

\begin{center}\begin{figure}[htbp]
\includegraphics[width=.75\textwidth]{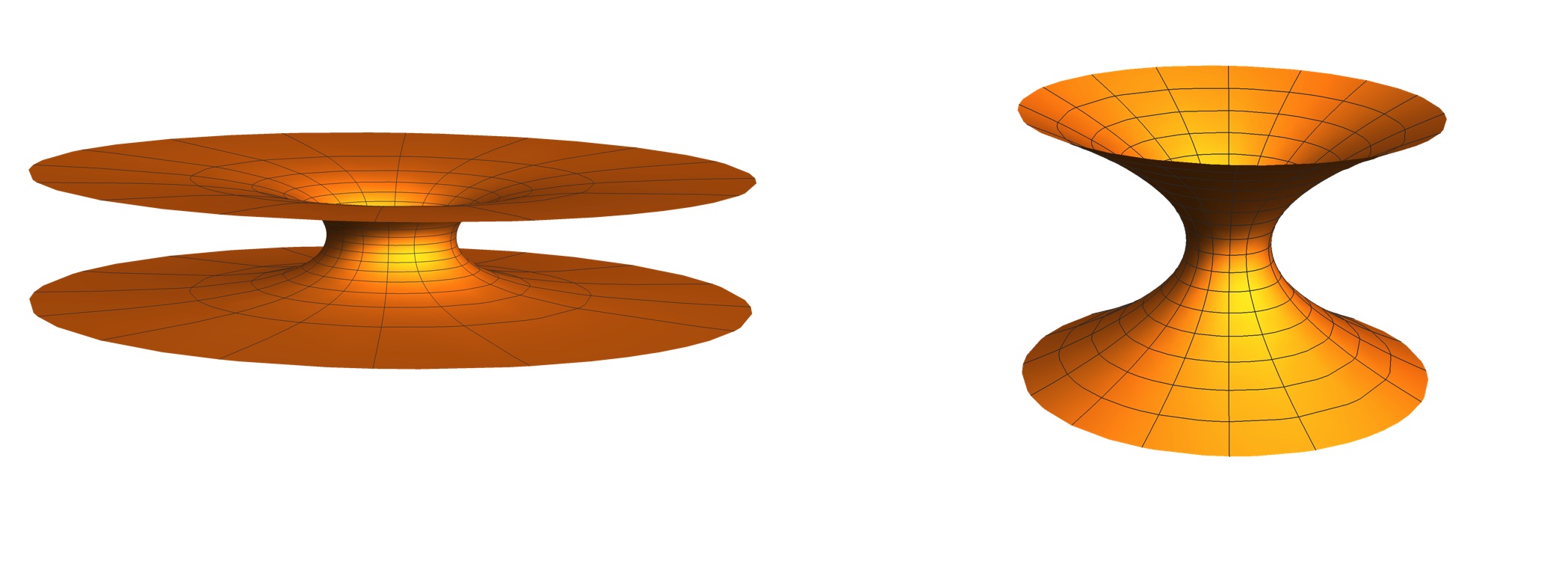}
\caption{Special catenoids with bounded and unbounded height (see Theorem \ref{heica}).}\label{fig:catenoide_reg}
\end{figure}\end{center}

\begin{remark}\label{re:sato2}
Let us make a brief comment about the orientation of the examples. In \cite{SaTo2}, the rotational surfaces are assumed to satisfy an elliptic Weingarten equation $H=F(H^2-K)$ for the orientation given by the exterior unit normal. Thus, when we write our elliptic Weingarten equation $\kappa_2=g(\kappa_1)$ as $H=F(H^2-K)$, the example constructed by Sa Earp and Toubiana in \cite{SaTo2} corresponds to $C^-_{\tau}$ with our notation. The other example $C^+_{\tau}$ that appears in our classification corresponds to the Sa Earp-Toubiana one for the choice $H=-F(H^2-K)$, after a change of orientation.
\end{remark} 

Let $\cG$ be defined by \eqref{eq:G}. If $|\cG(\Lambda)|=\8$ for both $\Lambda=\8$ and $\Lambda=b$, it follows from Theorem \ref{th:existence} that any rotational surface of the elliptic Weingarten class $\cW_g$ is a plane or a special catenoid. However, when $|\cG(\Lambda)|<\8$ for either $\Lambda=\8$ or $\Lambda=b$, we have again from Theorem \ref{th:existence} that $\cW_g$ contains rotational surfaces with singularities. As a consequence of the previous discussion, the corresponding orbits of these singular examples do not meet $\Gamma$. See Figure \ref{fig:minima}. We next describe in detail the resulting surfaces. 

\begin{center}\begin{figure}[htbp]
\includegraphics[width=0.75\textwidth]{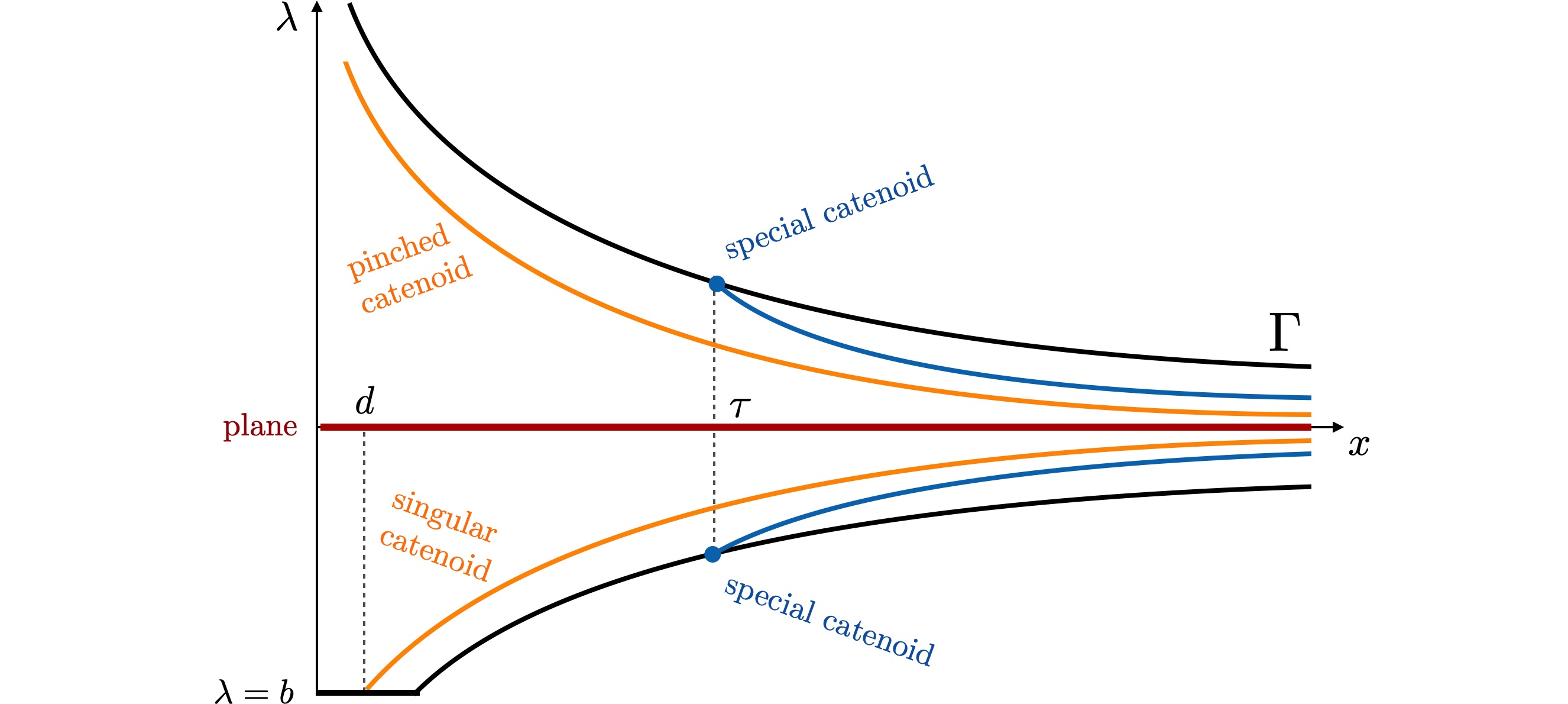}
\caption{Orbits corresponding to complete and singular rotational examples in $\cW_g$ when $\alfa=0$ and $b\neq-\8$.}\label{fig:minima}
\end{figure}\end{center}

To start, if we assume $|\cG(\8)|<\8$ and consider a fixed value $\landa_0>0$, it follows by Theorem \ref{th:existence} that the orbits with singularities in $\cR\cap \{\landa>0\}$, i.e., those lying in $\cR_{\rm sing}^+$, are given by initial conditions $(x_0,\landa_0)$ with 
$0<|\landa_0 x_0|\leq |\cG(\landa_0)|/|\cG(\8)|<1$. The elliptic Weingarten surfaces associated to these orbits are entire rotational graphs with an isolated singularity. Their profile curves are convex planar curves, and their unit normals point upwards. They can be seen as one half of a special catenoid after shrinking its neck to a point. See Figure \ref{fig:catenoide_sing}. We call such examples {\bf pinched special catenoids}. Their limit angle $\nu$ at the singularity is given by \eqref{singang}, and it can be checked that it can any value in $(-1,1)$.

A similar discussion can be made if we assume $|\cG(b)|<\8$, starting with $\landa_0 \in (b,0)$ and considering orbits in $\cR_{\rm sing}^-$, just by replacing $\8$ with $b$. This time, we have two possible situations. If $b=-\8$, we obtain again {\bf pinched special catenoids}, with a similar geometric structure to the previous ones, but this time their unit normals point downwards. They can also be parametrized in terms of their limit angle at the singularity.

If $b\neq-\8$, we obtain a rotational elliptic Weingarten graph over the exterior of a circle of a certain radius $d$. This surface is singular along its boundary circle, since its second fundamental form blows up there. Again, its profile curve is convex, and its unit normal points downwards. We call such surfaces {\bf singular special catenoids}. See Figure \ref{fig:catenoide_sing}. There is exactly one singular special catenoid for each value $d\in (0,-1/b ]$.

These pinched and singular special catenoids can have bounded or unbounded height, just as special catenoids do. See Theorem \ref{heica}.

\begin{center}\begin{figure}[htbp]
\includegraphics[width=\textwidth]{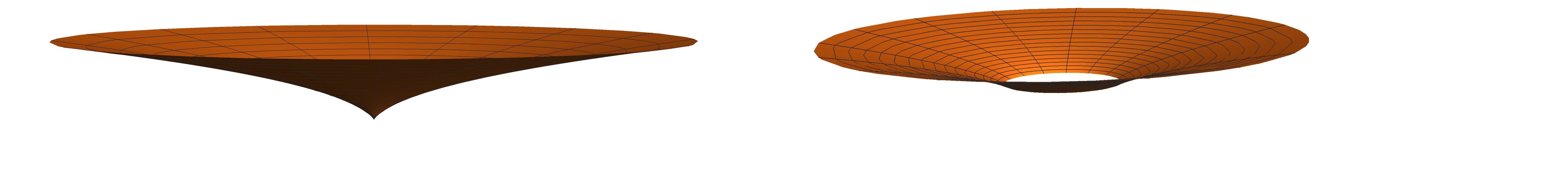}
\caption{Pinched and singular special catenoids.}\label{fig:catenoide_sing}
\end{figure}\end{center}

The theorem below summarizes the obtained result. In the statement, $\cG$ is the function defined in \eqref{eq:G}.
\begin{theorem}\label{th:minimal}
Let $\cW_g$ be an elliptic Weingarten class, with $\alfa=0$. Then, any rotational surface in $\cW_g$ is an open piece of one of the following surfaces: 
\begin{itemize}
\item A plane.
\item A special catenoid $C_{\tau}^+$ or $C_{\tau}^-$.
\item A pinched special catenoid. They exist if and only if either $|\cG(\8)|<\8$, or $b=-\8$ and $|\cG(b)|<\8$. 
\item A singular special catenoid. They exist if and only if  $b\neq-\8$ and $|\cG(b)|<\8$.
\end{itemize}
\end{theorem}


\subsection{Weingarten surfaces of CMC type ($\alfa\neq 0$ and $b<0$)} 

If the constants $\alfa,b$ of an elliptic Weingarten class $\cW_g$ satisfy $\alfa\neq 0$ and $b<0$, then the round spheres of principal curvatures $\alfa$ and the round cylinders of principal curvatures $\{0,f(0)\}$ are elements of $\cW_g$. For instance, the constant mean curvature equation $H=\alfa\neq 0$ is an example of this subclass of elliptic Weingarten equations.

So, let $\cW_g$ denote such an elliptic Weingarten class. Recall from Section \ref{sec:fases} that, in the phase space $\cR$ of \eqref{eq:auto1}, the point $(1/|\rho|,\rho)$ with $\rho:=f(0)$ corresponds to the round cylinder of $\cW_g$. By the condition $\alfa\neq 0$ and Lemma \ref{lem:xmonot}, it is easy to see that any orbit $(x(\landa),\landa)$ of \eqref{eq:auto1} intersects the boundary curve $\Gamma$ of $\cR$. Thus, we will see orbits of \eqref{eq:auto1} starting at some point $(x_0,\landa_0)\in \Gamma$.

Also, observe that by the behavior of the function $\cG(\landa)$ in Lemma \ref{lem:Gmonot}, and by Remark \ref{re:G}, any orbit satisfies $|\landa x(\landa)|\to \8$ if $\landa \to \alfa$; see also Figure \ref{fig:G}.

{\bf We will discuss first of all the case $\alfa>0$. }

By \eqref{monang}, we have $(\landa x(\landa))' >0$ if $\landa>\rho$. Since $|\landa x(\landa)|\to \8$ if $\landa \to \alfa$, we deduce that the orbit that starts at $(\tau,\landa_{\tau}) \in \Gamma$ with $\landa_{\tau}>\rho$ meets again $\Gamma\cap \{\landa>\alfa\}$ at a unique point $(\tau',\landa_{\tau'})$, with $\landa_{\tau'}\in (\alfa,\rho)$. See Figure \ref{fig:nominima}.

\begin{center}\begin{figure}[htbp]
\includegraphics[width=0.75\textwidth]{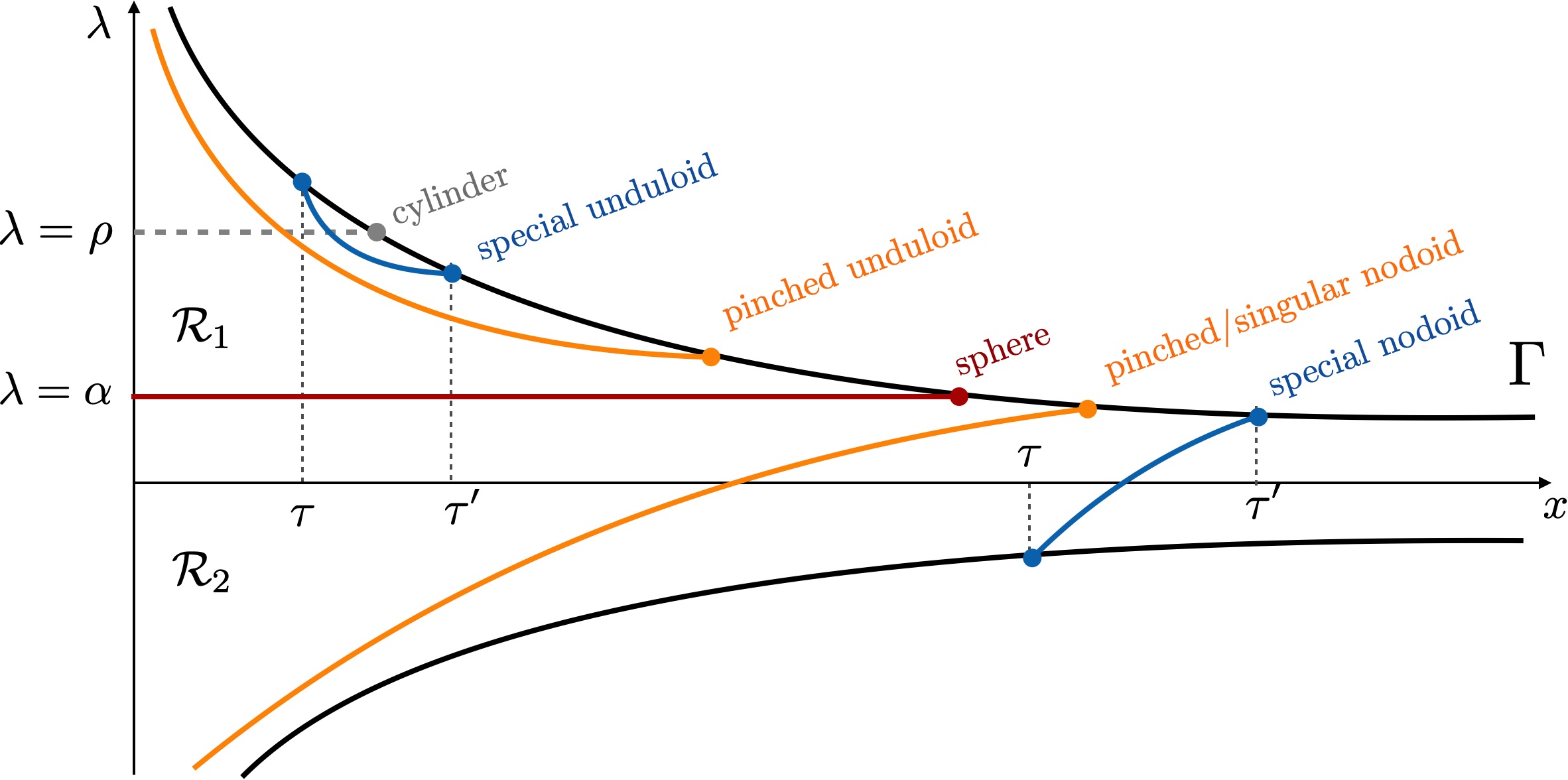}
\caption{Orbits corresponding to rotational examples in $\cW_g$ when $b<0$ and $\alfa>0$. The nodoid-type surfaces with singularities are pinched special nodoids if $b=-\8$ and singular special nodoids if $b>-\8$.} \label{fig:nominima}
\end{figure}\end{center}

These orbits generate the {\bf special unduloids} $\cO_{\tau}$ described by Sa Earp and Toubiana in \cite{SaTo}; see Figure \ref{fig:unduloideNodoide_reg}. They are complete, embedded surfaces, with a periodic profile curve whose minimum (resp. maximum) distance to the axis of rotation is achieved at $x=\tau$ (resp. $x=\tau'$). There is exactly one special unduloid $\cO_{\tau}$ for each $\tau\in (0,1/\rho)$, and they converge to the cylinder of $\cW_g$ as $\tau\to 1/\rho$.  The unit normal $N$ of $\cO_{\tau}$ is the interior one. 

\begin{center}\begin{figure}[htbp]
\includegraphics[width=0.5\textwidth]{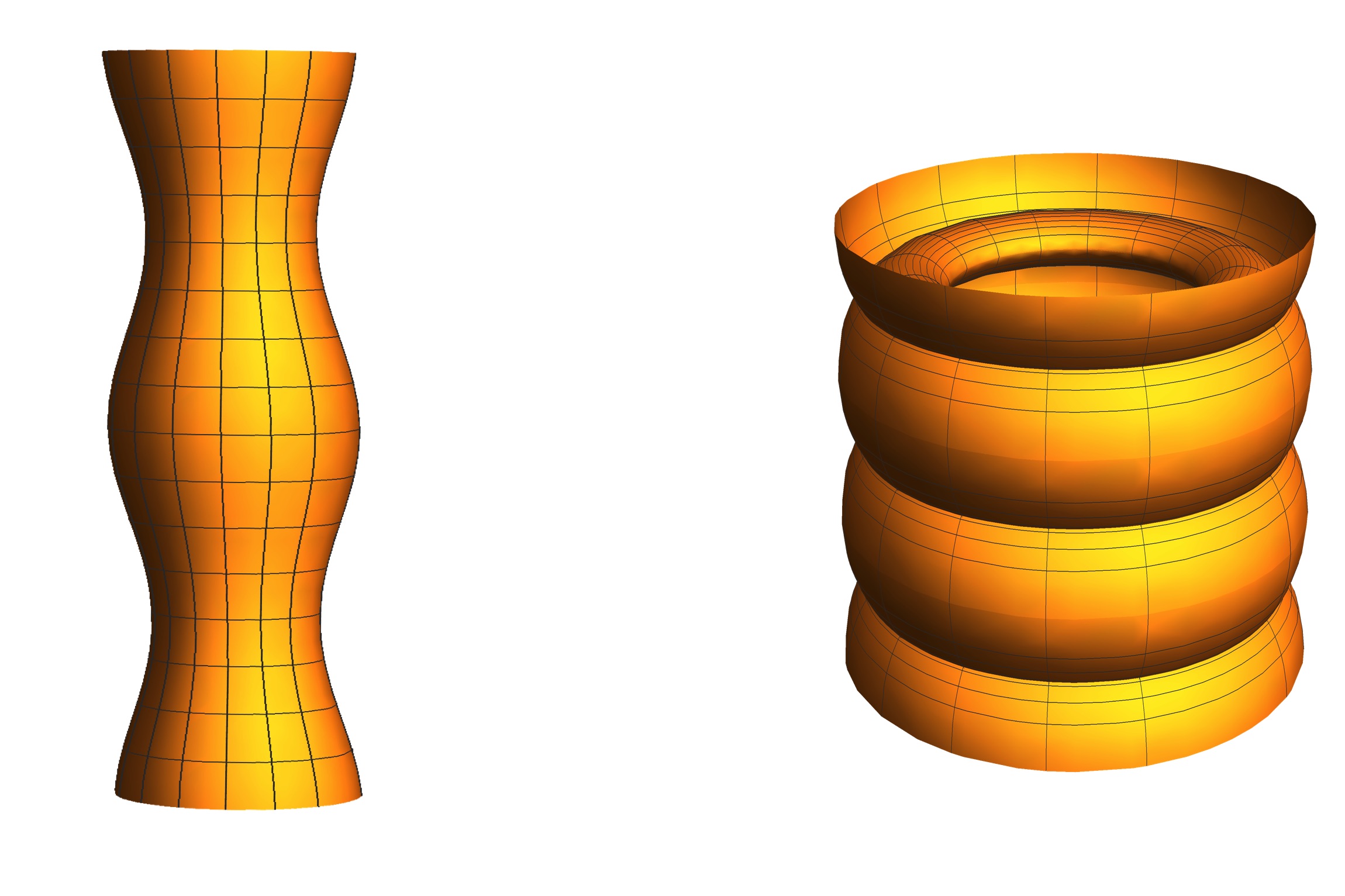}
\caption{Complete rotational examples in the non-minimal case. Left: special unduloid, right: special nodoid.}\label{fig:unduloideNodoide_reg}
\end{figure}\end{center}

Let now $\landa_1$ be the infimum of all the values $\landa_{\tau'}$ of the orbits of special unduloids. Clearly, $\landa_1\geq \alfa$, and if $\landa_1=\alfa$, the region $\cR^+:=\cR\cap \{\landa>\alfa\}$ is foliated by these orbits. In particular, by Theorem \ref{th:existence}, we have $|\cG(\8)|=\8$.

On the other hand, if $\landa_1>\alfa$, we have that for every $\landa_0\in (\alfa,\landa_1]$, the orbit that starts at $(1/\landa_0,\landa_0)\in \Gamma$ stays in $\cR^+$ for every $\landa>\landa_0$, see Figure \ref{fig:nominima}; so it lies in $\cR_{\rm sing}^+$, and by Theorem \ref{th:existence} we have $|\cG(\8)|<\8$. As a matter of fact, one can check that $\cG(\8)=\cG(\landa_1)$, but we will not use this property.

The orbits obtained in this way describe rotational surfaces with singularities in $\cW_g$ that we will call {\bf pinched special unduloids}. Any such surface has two isolated singularities, and a horizontal plane of symmetry (since the orbit meets $\Gamma$). See Figure  \ref{fig:unduloideNodoide_sing}. Topologically, they can be seen as embedded surfaces homeomorphic to $\S^2$, with two isolated singularities. The curvature of the surface around these two singularities is negative. These surfaces can be seen as {\em limit unduloids} as the necksize tends to zero. By symmetry, the limit angles at their singularities are given by $\{\nu,-\nu\}$ for some $\nu\in (0,1)$. By Proposition \ref{propang}, there exists at most one pinched special unduloid for each limit angle $\nu$, and one can check that there actually exists a (unique) pinched special unduloid for any $\nu\in [0,1)$. We denote such surface by $\cO^P_\nu$.

\begin{center}\begin{figure}[htbp]
\includegraphics[width=\textwidth]{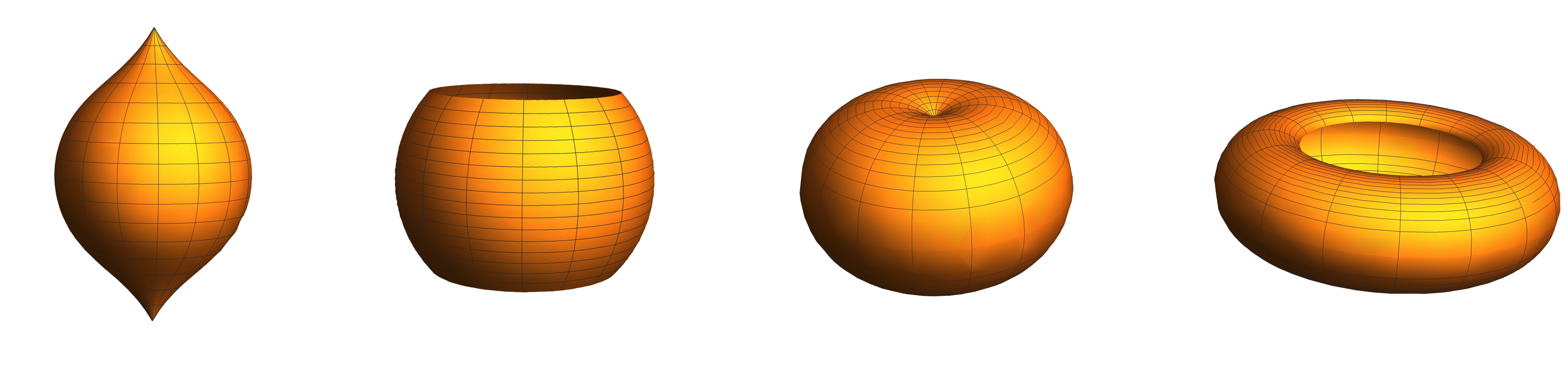}
\caption{From left to right: pinched special unduloid $\cO^P_\nu$, singular special unduloids $\cO_d^S$, pinched special nodoid $\cN^P_\nu $ and singular special nodoid $\cN^S_d$ .}\label{fig:unduloideNodoide_sing}
\end{figure}\end{center}

One should note that all the orbits considered up to this point foliate $\cR\cap \{\landa>\alfa\}$.

We consider next orbits lying in $\cR\cap \{\landa<\alfa\}$. Take $(\tau,\landa_{\tau})\in \Gamma$, with $\landa_{\tau}\in (b,0)$. The same arguments used above show now that the orbit that starts at $(\tau,\landa_{\tau})$ intersects $\Gamma\cap \{\landa>0\}$ at a unique point $(\tau',\landa_{\tau'})$, with $\landa_{\tau'}\in (0,\alfa)$. See Figure \ref{fig:nominima}. This orbit corresponds to one of the  {\bf special nodoids} $\cN_{\tau}$ described in \cite{SaTo}. They are given by a non-embedded periodic profile curve whose minimum (resp. maximum) distance to the axis of rotation is achieved at $x=\tau$ (resp. $x=\tau'$). See Figure \ref{fig:unduloideNodoide_reg}. There is one special nodoid $\cN_{\tau}$ for each $\tau\in (-1/b,\8)$.

Consider now $\landa_2$ as the supremum of the values $\landa_{\tau'}$ of the orbits of the special nodoids. Then, $\landa_2\in (0,\alfa]$, and if $\landa_2=\alfa$ these orbits foliate $\cR\cap \{\landa<\alfa\}$. In that case we must have $b=-\8$ and, by Theorem \ref{th:existence}, $|\cG(-\8)|=\8$.

Suppose next that $\landa_2<\alfa$. Then, for any $\landa_0\in [\landa_2,\alfa)$, the orbit that starts at $(1/\landa_0,\landa_0)$ must lie in $\cR_{\rm sing}^-$, and in particular, by Theorem \ref{th:existence}, we have $|\cG(b)|=\8$. See Figure \ref{fig:nominima}. There are two options.

If $b=-\8$, such orbits are defined for every $\landa<\landa_0$, and $x(\landa)\to 0$ as $\landa\to -\8$. Thus, the resulting surfaces have isolated singularities. We will call these surfaces {\bf pinched special  nodoids}. They have two isolated singularities, and they can be seen topologically as embedded $2$-spheres that are symmetric with respect to a horizontal plane; see Figure \ref{fig:unduloideNodoide_sing}. By symmetry, the limit angles at their singularities are given by $\{\nu,-\nu\}$ for some $\nu\in (0,1)$. Using Proposition \ref{propang} and Lemma \ref{lem:Gmonot}, it can be checked that for each $\nu\in [0,1)$ there exists a unique singular special unduloid with limit angles $\{\nu,-\nu\}$. We denote such surface by $\cN^P_\nu$.

On the other hand, if $b\in \R$, such orbits intersect the segment $\{\landa=b\}$ at a unique point, of the form $(d,b)$ with $d\in (0,1/|b|]$. See Figure \ref{fig:nominima}. The Weingarten surfaces described by these orbits are embedded surfaces diffeomorphic to an annulus, whose boundary is composed by two singular curves at the same distance $d>0$ from the rotation axis. Again, these surfaces are symmetric with respect to a horizontal plane. We call them {\bf singular special  nodoids}. See Figure \ref{fig:unduloideNodoide_sing}. It is easy to see that for each $d\in (0,1/|b|]$ there exists a unique singular special nodoid whose singular curves are at a distance $d$ from the rotation axis. We denote such surface by $\cN^S_d$.

Once here, we have considered all possible initial points in $\Gamma$ for orbits of \eqref{eq:auto1}. Since we know that any such orbit intersects $\Gamma$, we have found all possible rotational surfaces in $\cW_g$, when $\alfa>0$. 

{\bf We finally discuss the case $\alfa<0$}.

First of all, we should note that if $b=-\8$, then for $\alfa<0$ we obtain the same type of examples that for $\alfa>0$, but with their opposite orientation: see the comment regarding orientation after equation \eqref{eq:gorient}.

Suppose then that $b\in \R$, $b<\alfa<0$. Then, any orbit that starts at $(\tau,\landa_\tau)\in \Gamma\cap \{\landa>0\}$ intersects $\Gamma\cap \{\landa<0\}$ at a unique point $(\tau',\landa_{\tau'})$, with $\landa_{\tau'}\in (\alfa,0)$; see Figure \ref{fig:nominima2}. This orbit describes again a special nodoid $\cN_{\tau}$ with necksize $\tau\in (0,\8)$, see Figure \ref{fig:unduloideNodoide_reg}.

\begin{center}\begin{figure}[htbp]
\includegraphics[width=0.75\textwidth]{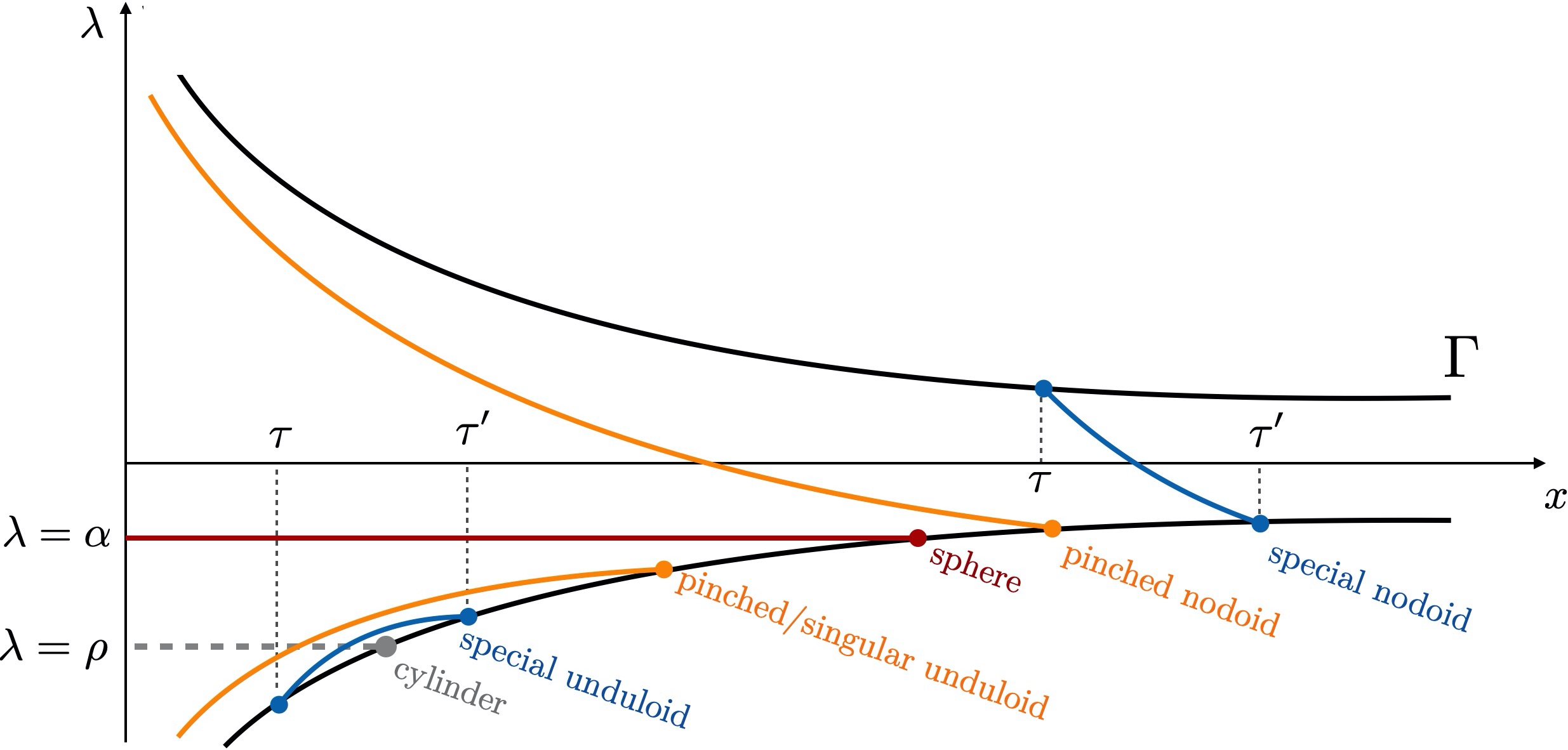}
\caption{Orbits corresponding to rotational examples in $\cW_g$ when $b<0$ and $\alfa<0$. The unduloid-type surfaces with singularities are pinched special unduloids if $b=-\8$ and singular special unduloids if $b\neq-\8$.}\label{fig:nominima2}
\end{figure}\end{center}

Let $\landa_1\in [\alfa,0)$ be the infimum of the values $\landa_{\tau'}$. If $\landa_1>\alfa$, then for any $\landa_0\in (\alfa,\landa_1]$, the orbit that starts at $(-1/\landa_0,\landa_0)$ stays in $\cR$ as $\landa\to \8$, and so it defines a pinched special nodoid $\cN^P_{\nu}$ similar to the ones previously described. See Figures \ref{fig:nominima2} and \ref{fig:unduloideNodoide_sing}

The orbits that begin at points of the form $(-1/\landa_0,\landa_0)\in \Gamma$ with $\landa_0\in (b,\rho)$ describe (complete) special unduloids $\cO_{\tau}$ (Figure \ref{fig:nominima2}). Their necksizes are given by $\tau=-1/\landa_0$ and take all values $\tau\in (-1/b,-1/\rho)$. In this $\alfa<0$ situation, the special unduloid $\cO_{\tau}$ is oriented with respect to its exterior unit normal. See Figure \ref{fig:unduloideNodoide_reg}.

Finally, at any point of the form $(d,b)$, with $d\in (0,1/|b|]$, there arrives a unique orbit, that must start as some point of the form $(-1/\landa_0,\landa_0)$, with $\landa_0<\alfa$. See Figure \ref{fig:nominima2}. In this way, we exhaust all possible orbits of the phase space $\cR$. These last orbits describe {\bf singular special unduloids}, denoted by $\cO^S_d$. These surfaces again have a horizontal plane of symmetry, are diffeomorphic to an annulus, and are bounded by two singular curves at the same distance $d>0$ from the rotation axis. Again, the surface has negative curvature near its singular set. See Figure \ref{fig:unduloideNodoide_sing}.

The previous cases actually describe all the possible behaviours in the phase space of $\cR$. We summarize the obtained classification the theorem below. In the statement, $\cG$ is the function defined in \eqref{eq:G}.

\begin{theorem}\label{th:nonminimal}
Let $\cW_g$ be a class of elliptic Weingarten surfaces, with $\alfa\neq 0$ and $b<0$. Then, any rotational surface in $\cW_g$ is an open piece of one of the following surfaces: 
\begin{itemize}
\item A sphere of principal curvatures $\alfa$.
\item A cylinder of principal curvatures $\{0,f(0)\}$. 
\item A special unduloid $\cO_\tau$.
\item A special nodoid $\cN_\tau$. 
\item A pinched special unduloid $\cO^P_\nu$ (they exist if and only if $\cG(\8)=\8$ and $\alfa>0$, or $b=-\8$, $\cG(-\8)=-\8$ and $\alfa<0$).
\item A pinched special nodoid $\cN^P_\nu$ (they exist if and only if $\cG(\8)=\8$ and $\alfa<0$, or $b=-\8$, $\cG(\8)=-\8$ and $\alfa>0$).
\item A singular special unduloid $\cO^S_d$ (they exist if and only if $b\neq-\8$ and $\alfa<0$).
\item A singular special nodoid $\cN^S_d$ (they exist if and only if $b\neq-\8$ and $\alfa>0$).
\end{itemize}
\end{theorem}

\begin{remark}\label{paraun}
For later use, let us remark that the special unduloids $\cO_\tau$ can be parametrized in terms of the necksize $\tau$, where $\tau\in (0,1/\rho)$ if $\alfa>0$, and $\tau\in (-1/b,-1/\rho)$ if $\alfa<0$. Similarly, the special nodoids $\cN_\tau$ are also classified by $\tau$, where this time $\tau\in (-1/b,\8)$ if $\alfa>0$, and $\tau\in (0,\8)$ if $\alfa<0$.
\end{remark}

\subsection{Summary of the classification}\label{sec:summary}

As we stated in the introduction, there are exactly 17 possible qualitative behaviors for a rotational elliptic Weingarten surface.

For the case of surfaces without singularities, we have 7 types: planes, spheres, cylinders, special unduloids, special nodoids, and special catenoids, which can have either bounded or unbounded height. All these examples form the classification of complete rotational elliptic Weingarten surfaces in \cite{SaTo,SaTo2}.

For the case of surfaces with singularities, we have 10 types: footballs, bracelets, pinched and singular special unduloids, pinched and singular special nodoids, and pinched and singular special catenoids, which can again have bounded or unbounded height.

In Theorem \ref{heica} we will classify when all these catenoids have bounded or unbounded height.

\section{The halfspace theorem}\label{sec:applications}

By the Hoffman-Meeks halfspace theorem \cite{HM}, any properly immersed minimal surface in $\R^3$ that lies on one side of a plane $\Pi$ must be itself a plane parallel to $\Pi$. The validity of a fully nonlinear version of this theorem, for the case of elliptic Weingarten surfaces, was studied by Sa Earp and Toubiana in \cite{SaTo2}. They showed that a halfspace theorem in this context is not generally true, since there exist special elliptic Weingarten catenoids properly embedded in a slab between two parallel planes (see also \cite{rafa} for additional examples of this type). On the other hand, they proved the halfspace theorem for some elliptic Weingarten classes $\cW_g$; see the discussion above Theorem \ref{sufici}. In this section we improve upon the results by Sa Earp and Toubiana, using the classification obtained in Theorem \ref{th:minimal}.

To start, we characterize the elliptic Weingarten classes $\cW_g$ for which the special catenoids in $\cW_g$ have unbounded height. For the case of a linear elliptic Weingarten relation $\kappa_2 = a \kappa_a$ with $a<0$, this was done by López and Pámpano in \cite{rafa}.

Recall from Section \ref{sec:a0} that, on any elliptic Weingarten class $\cW_g$ with $\alfa=0$, there are two types of special catenoids: $C_{\tau}^+$, with an interior unit normal, and $C_{\tau}^-$, with an exterior one. The principal curvature $\landa$ of $C_{\tau}^+$ (resp. of $C_{\tau}^-$) is positive (resp. negative). Theorem \ref{heica} below applies to $C_{\tau}^+,C_{\tau}^-$, but also to the pinched and singular catenoids of $\cW_g$ classified in Theorem \ref{th:minimal}, in case they exist. 

\begin{remark}\label{rem:height}
The \emph{if} part in Theorem \ref{heica} below is proved in \cite[Theorem 1]{SaTo2} and so it will be omitted here. We remark that the result in \cite{SaTo2} is stated only for the case of (complete) special catenoids, and using a different notation for the Weingarten equation, but the proof works also for the case of pinched and singular catenoids of $\cW_g$, since it only deals with the behavior at infinity of such surfaces.
\end{remark}

\begin{theorem}\label{heica}
Let $\cW_g$ be an elliptic Weingarten class with $\alfa=0$. Let $m:=g'(0)<0$. Then, the rotational surfaces in $\cW_g$ with $\landa>0$ (resp. $\landa<0$)  have unbounded height if and only if $-1 \leq m<0$ (resp. $m\leq -1$). 
\end{theorem}
\begin{proof}

Let $\Sigma$ be a rotational surface in $\cW_g$ with $\landa>0$ (the case $\landa<0$ is analogous). Then, $f$ has a finite right derivative at the origin, given by $m$. Assume that $m<-1$. For $t>0$ near zero we can write
\begin{equation}\label{ineqf}
f(t)< f_2(t)<0,
\end{equation}
where $f_2(t):=a_2 t$ for $a_2\in (m,-1)$.

Given a fixed $(x_0,\landa_0)\in\cR$ with $\landa_0>0$, consider $( x(\landa), \landa)$ and $( x_2(\landa), \landa)$ the orbits passing through $(x_0,\landa_0)$ for the systems \eqref{eq:auto1} associated to the functions $f$, $f_2$, respectively. Since these orbits are given by \eqref{eq:soledo}, if we choose $\landa_0$ small enough so that \eqref{ineqf} holds in $(0,\landa_0)$, we obtain for every $\landa\in (0,\landa_0)$ that $x(\landa) < x_2(\landa)$.

Let now $\Sigma_2$ be the rotational surface associated to the orbit $(x_2(\landa), \landa)$, where $\landa \in (0,\landa_0]$. By choosing $r_0>x_0$ large enough, both surfaces $\Sigma,\Sigma_2$ can be seen as radial graphs around the $z$-axis, with associated graphing functions $u_1(r),u_2(r)$, where $r$ varies in $[r_0,\8)$. For any such $r>r_0$, we have $$u_i'(r)^2 = \frac{\landa_i^2 r^2}{1-\landa_i^2 r^2},$$ where $\landa_i$, $i=1,2$, are the values at $r$ of the principal curvature $\landa$ of each of the graphs $z=u_i(r)$. Using here now that $x(\landa) \leq x_2(\landa)$ 
for every $\landa\in (0,\landa_0)$ for their associated orbits, we deduce that $\landa_1\leq \landa_2$, and hence $$u_1'(r)^2\leq u_2'(r)^2.$$ Therefore, up to a vertical translation and possibly a $180º$-rotation around the $x$-axis, the graph $\Sigma$ lies between $z=0$ and $\Sigma_2$ outside a sufficiently large compact set. In \cite{rafa}, López and Pámpano proved that the height of a rotational surface satisfying $\mu=a\, \landa$ for some $a<0$ is bounded if and only if $a<-1$. Thus, $\Sigma_2$ has bounded height, since so does $\Sigma$. As explained in Remark \ref{rem:height}, the converse was proved in \cite{SaTo2}.
\end{proof}

\begin{definition}
We say that an elliptic Weingarten class $\cW_g$ satisfies the halfspace property if the only properly immersed surfaces $\Sigma$ of $\cW_g$ that lie on one side of a plane $\Pi$ in $\R^3$ are the planes parallel to $\Pi$. \end{definition}

\begin{corollary}\label{neces}
A necessary condition for $\cW_g$ to satisfy the halfspace property is that $g(0)=0$ and $g'(0)=-1$.
\end{corollary}
\begin{proof}
If $g(0)\neq 0$, there exist round spheres in the class $\cW_g$, what contradicts the halfspace property. If $g(0)=0$ and $g'(0)\neq -1$, it follows from Theorem \ref{heica} that one of the two types of special catenoids $C_{\tau}^+$, $C_{\tau}^-$ of $\cW_g$ have bounded height, and thus it contradicts again the halfspace property.
\end{proof}

As regards sufficient conditions, Sa Earp and Toubiana proved in \cite[Corollary 1.1]{SaTo2} that the following set of conditions is sufficient for the validity of the halfspace property in an elliptic Weingarten class $\cW_g$:

\begin{enumerate}
\item
$g(0)=0$ and $g'(0)=-1$. (Note that these are necessary, by Corollary \ref{neces}).
\item
$t+g(t)$ does not change sign in $[0,\8)$.
 \item
$b=g(\infty)=-\infty$.
\end{enumerate}
The theorem below improves this result, by providing a more general sufficient condition. 

\begin{theorem}\label{sufici}
Let $\cW_g$ be a class of elliptic Weingarten surfaces, with $g(0)=0$ and $g'(0)=-1$. If, additionally, $t+g(t)$ does not change sign in some small interval $(0,\ep)$, $\ep>0$, then $\cW_g$ satisfies the halfspace property.
\end{theorem}
\begin{proof}
We use the basic strategy of the classical proof of the halfspace theorem for minimal surfaces, together with our description of rotational examples of $\cW_g$ in Theorems \ref{th:minimal} and \ref{heica}. We will merely outline some of the most standard details of the argument, and focus on the new ingredients of the proof. 

Arguing by contradiction, assume that $\Sigma\in \cW_g$ is properly immersed in $\R^3$ and lies in, say, the halfspace $\{z>0\}$, but not in any other halfspace $\{z>z_0>0\}$. 

We will assume that $t+g(t)\geq 0$ in $(0,\ep)$; the proof is similar if $t+g(t)\leq 0$. 
Let $G_1$ be the lower half of the special catenoid  $C^-_{\tau'}$ of the class $\cW_g$, with rotation axis the $z$-axis, for some necksize $\tau'>0$.  That is, $G_1$ is the piece of $C^-_{\tau'}$ that lies below its neck. Then, $G_1$ is a radial graph $z=u_1(x,y)$ of a function $u_1$ defined on the exterior of a disc of radius $\tau'$ in the plane $z=0$, and it lies in $\cW_g$ for the orientation given by its upwards-pointing normal vector. The associated orbit $\gamma_1$ of $G_1$ in $\cR^-$ ends up at the point $(\tau',-1/\tau')\in \Gamma$.

Now consider, for any $\sigma\in (0,1)$, the  orbit $\gamma_\sigma=(x_\sigma(\landa),\landa)$  passing through the point $(\tau',-\sigma/\tau')\in \cR^-$. These orbits give rise to a family of upwards-oriented rotational graphs $\{G_\sigma\}_{\sigma\in (0,1]} \subset \cW_g$, with the $z$-axis as their rotation axis. Up to translation, each such $G_{\sigma}$ is a graph $z=u_{\sigma}(x,y)$ of a function $u_{\sigma}$ defined over a domain of the $z=0$ plane, that is either the exterior of a disc around the origin or the punctured plane, and so that $u_{\sigma}=0$ at the boundary of this domain. The \emph{meridian} and \emph{parallel} principal curvatures of the rotational graphs $G_{\sigma}$ satisfy $\landa_\sigma<0<\mu_\sigma$, and 
 \begin{equation}\label{lalasi}
\landa_\sigma=g(\mu_{\sigma}).
\end{equation}
Let us describe in more detail the family $G_{\sigma}$, using our discussion in Section \ref{sec:classification}. If $b=-\8$ and $\cG(-\8)=\8$ all the surfaces $\{G_\sigma\}_{\sigma\in (0,1]}$ are the lower halves of the special catenoids $C^-_{\tau}$ of $\cW_g$, with $\tau\in (0,\tau')$. They converge to the horizontal plane $z=0$ punctured at the origin as $\sigma\to 0$ (equivalently, as $\tau\to 0$).

If $b=-\8$ and $\cG(-\8)<\8$, there exists $\sigma_0\in (0,1)$ such that $\{G_\sigma :\sigma>\sigma_0\}$ are the lower halves of the special catenoids $C_{\tau}^-$ with $\tau\in (0,\tau')$, while $\{G_\sigma : \sigma\leq \sigma_0\}$ is the family of upwards-oriented pinched special catenoids of $\cW_g$. When $\sigma\to 0$, these pinched catenoids $G_\sigma$ converge uniformly on compact sets to the plane $z=0$ punctured at the origin.

Finally, if $b\neq -\8$, then there exists again $\sigma_0\in (0,1)$ such that $\{G_\sigma :\sigma>\sigma_0\}$ are the lower halves of the special catenoids $C_{\tau}^-$ with $\tau\in (0,\tau')$, but this time $\{G_\sigma : \sigma\leq \sigma_0\}$ is the family of upwards-oriented singular special catenoids of $\cW_g$. As $\sigma\to 0$, the distance $d$ from the singular curve of $G_\sigma$ to the rotation axis (the $z$-axis) also tends to zero. The $G_{\sigma}$ converge again to the plane $z=0$ punctured at the origin, uniformly on compact sets.

Consider next the value $\ep>0$ as in the statement of Theorem \ref{sufici}. The behaviour of the orbits $\gamma_\sigma$ implies that there exists $R>0$ such that, for every $\sigma\in (0,1)$, it holds
\begin{equation}\label{laro}
|\landa_\sigma(r)| < |\landa_1(r)|<\ep, \ \text{for all $r>R$,}
\end{equation}
where $r$ is the distance to the rotation axis. See Figure \ref{fig:semiespacio}. 
\begin{center}\begin{figure}[htbp]
\includegraphics[width=\textwidth]{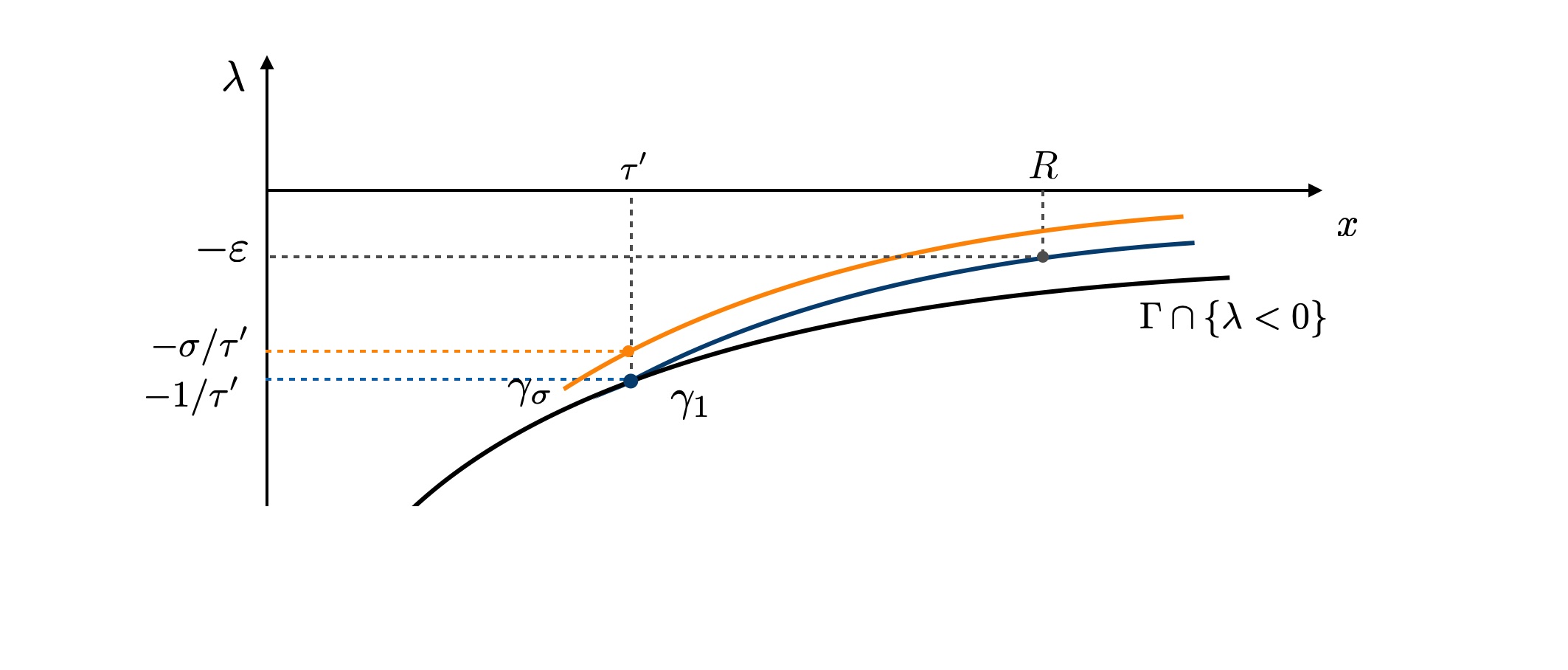}\vspace{-1.2cm}
\caption{The orbits $\gamma_\sigma$, $0<\sigma\leq 1$.}\label{fig:semiespacio}
\end{figure}\end{center}

In particular, since we are assuming $t+g(t)\geq 0$ in $(0,\ep)$, it follows from \eqref{lalasi} and \eqref{laro} that the graphs $G_\sigma$ have non-negative mean curvature outside the vertical cylinder in $\R^3$ of radius $R$.

Take now $c>0$ small enough so that $\Sigma \cap G_1^{\delta}=\emptyset$ for any $\delta\in (0,c]$, where $G_1^\delta:=G_1 + (0,0,\delta)$ (recall that $\Sigma$ is proper in $\R^3$, while $G_1$ lies in $\{z\leq 0\}$, with $\parc G_1$ contained in the $z=0$ plane). Note that, since $g'(0)=-1$, we have by Theorem \ref{heica} that the graphs $\{G_\sigma\}_{\sigma\in (0,1]}$ are not bounded from below at infinity. Then, using the standard argument of the Hoffman-Meeks halfspace theorem together with the already discussed fact that the family $\{G_\sigma\}_{\sigma\in (0,1]}$ converges to the punctured $z=0$ plane as $\sigma\to 0$, we deduce that for any $\delta\in(0,c]$, if we denote $G_{\sigma}^\delta:=G_{\sigma} + (0,0,\delta)$, then there exists $\sigma(\delta)\in (0,1)$ such that $\Sigma\cap G_{\sigma}^\delta$ is empty for all $\sigma>\sigma(\delta)$, but there exists some $p(\delta)\in\Sigma\cap G_{\sigma(\delta)}^\delta$. Note that $\Sigma,G_{\sigma(\delta)}^\delta$ are necessarily tangent at this \emph{first contact point} $p(\delta)$. If they have the same unit normal at $p(\delta)$, we reach a contradiction with the maximum principle within the elliptic Weingarten class $\cW_g$. Thus, $\Sigma,G_{\sigma(\delta)}^\delta$ have opposite unit normals at $p(\delta)$, and in particular $\Sigma$ must be downwards oriented at $p(\delta)$.

Let $R>0$ be the constant in \eqref{laro}. Let $D_{R}:= \{(x,y,0): x^2+y^2\leq R^2\}$ and $\Sigma_{R}:=\{p\in \Sigma: \pi(p)\in D_{R}\}$, where $\pi(x,y,z):=(x,y,0)$. Since $\Sigma$ is proper and lies in $\{z>0\}$, we have ${\rm dist}(\Sigma_R, D_R )=:\rho_0>0$. Take $\delta<\rho_0$, and note that, with the above notations, $G_{\sigma}^\delta$ always lies in the half-space $\{z\leq \delta\}$. Thus, $\Sigma_R\cap G_{\sigma(\delta)}^\delta$ is empty, i.e., $\pi(p(\delta))$ lies outside $D_R$.

In particular,  for $\delta>0$ sufficiently small, and due to the definition of the constant $R>0$, the graph $G:=G_{\sigma(\delta)}^\delta$ has non-negative mean curvature around its contact point $p(\delta)$ with $\Sigma$. Recall also that $G$ is upwards-oriented and $\Sigma$ lies above it. 

 Let $\Sigma^*$ denote $\Sigma$ with its opposite orientation. Then, $G,\Sigma^*$ have the same unit normal at $p(\delta)$, and $\Sigma^*$ lies above $G$. By the comparison principle, at $p(\delta)$, the largest (resp. smallest) principal curvature of $\Sigma^*$ is greater than the largest (resp. smallest) one of $G$. Thus, if $\kappa_1^{\Sigma}\geq 0\geq  \kappa_2^{\Sigma}$ are the principal curvatures of $\Sigma$ with its initial orientation (the one for which $\Sigma\in \cW_g$), we have $\kappa_1^{\Sigma}\leq -\landa^G <\ep$ at $p(\delta)$, by \eqref{laro}. Using now that $\kappa_2^{\Sigma}=g(\kappa_1^{\Sigma})$ and that $t+g(t)\geq 0$ if $t\in [0,\ep)$, we conclude that $\Sigma$ has non-negative mean curvature around $p(\delta)$. Thus $\Sigma^*$ has non-positive mean curvature, $H_{\Sigma^*}\leq 0$, around $p(\delta)$. Since $G$ has $H\geq 0$ and lies below $\Sigma^*$, this contradicts the interior maximum principle (see Lemma 1 in \cite{Sch}). This contradiction completes the proof.   
\end{proof}

\begin{definition}
We say that $g$ has \emph{finite order} at the origin if, for some $n> 1$, $g$ is of class $C^n$ around $0$, and $g^{(n)}(0)\neq 0$.
\end{definition}

\begin{corollary}\label{th:order}
Let $\cW_g$ be an elliptic Weingarten class, and assume that $g$ has finite order at the origin. Then $\cW_g$ satisfies the halfspace property if and only if $g(0)=0$ and $g'(0)=-1$.
\end{corollary}
\begin{proof}
The \emph{only if} part follows from Corollary \ref{neces}. The \emph{if} part is a consequence of Theorem \ref{sufici}, since the finite order hypothesis implies that $t+g(t)$ does not change sign in some $(0,\ep)$.
\end{proof}


\section{Isolated singularities of elliptic Weingarten surfaces}\label{sec:isolated}

We next prove that isolated singularities of elliptic Weingarten graphs in $\R^3$, not necessarily rotational, are always bounded.
\begin{theorem}\label{th:isosin}
If a graph $z=u(x,y)$ over a punctured disk $D^*\subset \R^2$ satisfies an elliptic Weingarten equation, then $u$ is bounded around the puncture.
\end{theorem}
\begin{proof}
Let $\cW_g$ be the elliptic Weingarten class to which the graph $\Sigma\equiv z=u(x,y)$ belongs, and let $N$ be the unit normal of $\Sigma$. Without loss of generality, we assume that $N$ points upwards. Let $D^*$ be a disk of radius $r_0>0$ centered and punctured at the origin, and take constants $m,M$ such that $m\leq u \leq M$ on the boundary $\parc D^*=\S^1(r_0)$. We will divide our discussion into several cases, depending on the sign of the umbilical constant $\alfa$ of $\cW_g$.

Assume $\alfa>0$. It is clear that $u$ is bounded from above in $D^*$. Indeed, let $S_{\alfa}\in \cW_g$ be a round sphere of principal curvatures equal to $\alfa$. Place $S_{\alfa}$ so that it lies on the open halfspace $z>M$ of $\R^3$, and $S_{\alfa}\cap\Sigma$ is empty. If $u$ is not bounded from above, by translating $S_{\alfa}$ horizontally towards $\Sigma$ we will reach an interior first contact point between both surfaces, where the unit normals agree (note that the unit normal of $S_{\alfa}$ is the interior one, and the first contact point happens below the equator of $S_{\alfa}$). This contradicts the maximum principle and shows that $u$ is bounded from above.

We now prove that $u$ is bounded from below, adapting an argument from \cite{ER}. Let $\cO_{\tau_0}$ be the special unduloid of necksize $\tau_0$ of the Weingarten class $\cW_g$, with rotation axis given by the $z$-axis. Its unit normal points to the interior region bounded by $\cO_{\tau_0}$. Let $\cO_{\tau_0}^0$ be the compact piece of $\cO_{\tau_0}$ whose boundary consists of two consecutive circles at which the distance from $\cO_{\tau_0}$ to the rotation axis attains its maximum, $\tau_0'$. By taking $r_0$ smaller if necessary, assume that $r_0<\tau_0$. Next, translate $\cO_{\tau_0}^0$ downwards, so that its upper boundary lies in the plane $z=m-\ep$ for some $\ep>0$ (here $m={\rm min} \, u$ on $\S^1(r_0)$). For each $\tau\in (0,\tau_0]$, let $S_{\tau}$ be the surface obtained from the special unduloid $\cO_{\tau}$ via the above process; these surfaces exist, see Remark \ref{paraun}. It follows from the phase space analysis that when we decrease $\tau$, the radius $\tau'$ of its upper boundary (and also, by symmetry, of its lower boundary) in the plane $z=m-\ep$ increases. Note that both $S_{\tau_0}\cap \Sigma$ and $S_{\tau}\cap \parc \Sigma$ are empty. Also, the heights of the compact family $S_{\tau}$ are uniformly bounded for $\tau\in (0,\tau_0]$. Assume now that $u$ is not bounded from below around the origin. Then, by decreasing the necksize $\tau\to 0$ we would reach a first interior contact point of some element of the $S_{\tau}$-family with $\Sigma$, at which both surfaces have the same unit normal. This contradicts again the maximum principle. Thus, $u$ is bounded if $\alfa>0$.

If $\alfa<0$, the argument is similar. This time, we obtain that $u$ is bounded from below using the sphere $S_{\alfa}\in \cW_g$, which is now oriented by its exterior normal. In order to obtain that $u$ is bounded from above, we use as comparison objects adequate compact pieces of the special nodoids $\cN_{\tau}$ instead of the special unduloids $\cO_{\tau}$; see again Remark \ref{paraun}. We omit the details, as the process is analogous.

Finally, assume that $\alfa=0$. Let $C$ be a compact piece of a special catenoid $C_{\tau}^+$ of $\cW_g$, with rotation axis given by the $z$-axis, and with its interior orientation, and so that $\parc C$ is the union of two circles in horizontal planes. Let $2h>0$ be the distance between these two planes. By taking $r_0$ smaller if necessary, assume that $r_0<h$.

After an isometry, place $C$ so that: (i) the two boundary components of $\parc C$ lie in the vertical planes $x=h$ and $x=-h$; (ii) $C$ lies in the open half-space $z>M$ (recall that $M= {\rm max} \, u$ on $\S^1(r_0)$), and (iii) $C\cap \Sigma$ is empty. If $u$ is not bounded from above near the origin, translating now $C$ towards $\Sigma$ horizontally in the $y$-direction we reach a contradiction with the maximum principle.

The same argument using a compact piece of a special catenoid $C_{\tau}^-$ of $\cW_g$ with its exterior orientation shows that $u$ must also be bounded from below near the origin. This completes the proof.
\end{proof}

We next consider a global classification problem for elliptic Weingarten surfaces with isolated singularities. Motivated by \cite{GHM}, we introduce the following definition.

\begin{definition}
A \emph{peaked elliptic Weingarten sphere} is a closed convex surface
$S\subset \R^3$ (i.e. the boundary of a bounded convex set of
$\R^3$) that is regular everywhere except at a finite set $p_1,\dots, p_n\in S$, 
and such that $S\setminus
\{p_1,\dots, p_n\}$ is an elliptic Weingarten surface. 
We call $p_1,\dots, p_n$ the \emph{singularities} of $S$.
\end{definition}

We note that there exist peaked spheres with constant curvature $K=1$ (which is an elliptic Weingarten equation), with any number $n\geq 2$ of singularities, see \cite{GHM}.

\begin{theorem}\label{2sing}
Let $S$ be a peaked elliptic Weingarten sphere with $n\leq 2$ singularities. Then either $n=0$ and $S$ is a round sphere, or $n=2$ and $S$ is a rotational football $\cF_{\nu}$ (see Theorem \ref{th:Kpositiva}).
\end{theorem}
\begin{proof}
If $n=0$, the theorem follows directly from the Alexandrov reflection principle \cite{A}, or alternatively, from the Hopf uniqueness theorem for immersed elliptic Weingarten spheres in \cite{GM3}. From now on, assume that $n\in \{1,2\}$. If $n=2$, let $\{p_1,p_2\}$ be the two singularities of $S$. If $n=1$, we let $p_1$ be the singularity of $S$, and $p_2$ be another arbitrary point.

We will use the Alexandrov reflection principle. Let $L$ be the line passing through $p_1$ and $p_2$, and let $\Pi_0$ be a plane parallel to $L$. For definiteness, assume after a rigid motion that $L$ is the $z$-axis of $\R^3$, that $\Pi_0$ is a plane $x={\rm const}$, and that $p_1=(a,0,0)$, $p_2=(-a,0,0)$ for some $a>0$. Take a plane $x=c$, with $c<0$, away enough so that it does not intersect $S$, and start moving it towards $S$ until reaching a first contact point of $S$ with some plane $x=c_0<0$. For each $c\in (c_0,0)$ beyond that moment, let $S_c^*$ denote the reflection of $S\cap \{x\leq c\}$ with respect to $x=c$. By applying the Alexandrov reflection principle in the standard way, we end up having one of the following three possibilities.

{\bf (a)} There exists a first contact point of $S_c^*$ with $S\cap \{x\geq c\}$ at some regular point of $S\cap \{x\geq c\}$, for some $c<0$. In that case, $x=c$ would be a plane of symmetry of $S$, a contradiction with the fact that $p_1,p_2$ are singularities of $S$ (they lie on the plane $x=0$). So this case does not happen.

{\bf (b)}  There exists a first contact point of $S_c^*$ with $S\cap \{x\geq c\}$ at one of the isolated singularities of $S$, for some $c<0$. We will see below that this situation is not possible.

{\bf (c)} $S_0^*$ is contained in the interior region bounded by $S\cap \{x\geq 0\}$. In that situation, if we now apply the reflection principle starting from $x=c$ with $c>0$ large enough, using items {\bf (a)}, {\bf (b)} we deduce that $x=0$ must be a plane of symmetry of $S$.

As we can do this process with respect to any plane $\Pi$ parallel to $L$, we deduce that $S$ is rotationally symmetric. By the classification of rotationally symmetric elliptic Weingarten surfaces in Section \ref{sec:classification}, the only such compact, convex surfaces are round spheres and the footballs $\cF_{\nu}$ of Theorem \ref{th:Kpositiva}.

So, to complete the proof of Theorem \ref{2sing} we only need to show that the situation described by item {\bf (b)} is not possible. In the conditions of {\bf (b)}, assume that $S_c^*$ has a first contact point with $S\cap \{x\geq c\}$ at a singularity $p$ of $S$. In adequate local coordinates $(x_1,x_2,x_3)$, we can assume that $p=(0,0,0)$, that $S_c^*$ and $S$ are graphs $x_3= u^*(x_1,x_2)$ and $x_3=u(x_1,x_2)$ over a punctured disk $\Omega^*$ of a small radius $\ep_0$ around the origin, with both $D^2 u, D^2 u^*$ positive definite on $\Omega^*$, and that $u\leq u^*$ in $\Omega^*$. Note that $u=u^*=0$ at the origin. Moreover, we can also assume that $D u^*(0,0)=(0,0)$. From these conditions, since $u^*$ is regular at the origin, we deduce by convexity that $u$ has a unique support plane at the origin, i.e., $u$ is of class $C^1$ at the origin, with $Du(0,0)=(0,0)$.

Since $u^*>u$ on $\S^1(\ep_0)=\parc \Omega$, we can slightly tilt the graph $x^3 = u^*(x_1,x_2)$ by making a small rotation around the $x_1$-axis, so that still $u^*>u$ on $\S^1(\ep_0)$ for the function obtained after this tilting, that we still denote by $u^*$. Note that now $u^*$ is not greater than $u$ anymore around the origin. Next, translate the graph $x_3 = u^*(x_1,x_2)$ vertically until it is above the graph of $u$, and start translating it back downwards until reaching a first contact point with the graph of $u$. This point does not lie above the boundary $\S^1(\ep_0)$ or the origin, since $u^*$ is not above $u$ near the origin, but $u(0,0)=u^*(0,0)=0$. This situation contradicts the maximum principle (note that the \emph{tilted} graph for $u^*$ is a solution of the same Weingarten equation as the original $u^*$), and completes the proof.
\end{proof}

\section{The non-elliptic case}\label{sec:yau}

A famous theorem by Chern \cite{Ch0} proves that if the principal curvatures of a $C^2$ ovaloid $S$ in $\R^3$ satisfy, in some order, a Weingarten equation $\kappa_1=h(\kappa_2)$, where $h$ is a decreasing function, then $S$ is a round sphere. Chern remarked that the hypothesis that $h$ is decreasing cannot be removed, since any rotational ellipsoid in $\R^3$ satisfies a Weingarten equation of the type 
 \begin{equation}\label{yau}
\kappa_2 = c \kappa_1^3,
\end{equation} for some $c>0$. Inspired by this situation, Yau \cite{Yau} (Problem 58, p. 682), asked whether any compact surface in $\R^3$ whose principal curvatures $\kappa_1,\kappa_2$ satisfy \eqref{yau} in some order must be a rotational ellipsoid. One should note here that the \emph{meridian} and \emph{parallel} principal curvatures $\mu,\landa$ of a rotational surface do not make sense anymore for arbitrary surfaces; this explains the \emph{in some order} comment for $\kappa_1,\kappa_2$. Also, observe that if a compact surface satisfies \eqref{yau}, then it has non-negative Gaussian curvature, and thus it has genus zero. 

Yau's question has an affirmative answer in the real analytic case, see \cite{KS,Si}. Indeed, by a striking classical theorem of Voss \cite{Vo}, any real analytic Weingarten surface $\Sigma$ of genus zero in $\R^3$ is a rotational sphere. (See \cite{GM5} for a version of Voss' theorem in the context of overdetermined PDE problems in the plane). Thus, if such $\Sigma$ satisfies \eqref{yau}, the \emph{meridian} and \emph{parallel} principal curvatures $\mu,\landa$ of $\Sigma$ verify either $\mu = c\landa^3$ or $\mu = c\landa^{1/3}$ at every point. Kühnel and Steller proved that if a real analytic rotational sphere satisfies any of these two equations, then it is an ellipsoid. This fact together with Voss' theorem provides a positive answer to Yau's question in the real analytic case. An alternative proof was later given by Simon \cite{Si}.

In contrast, we prove next that for $C^2$ surfaces, the answer to Yau's problem is negative.

\begin{theorem}\label{th:yau}
There exist compact surfaces of class $C^2$ in $\R^3$ other than ellipsoids whose principal curvatures $\kappa_1,\kappa_2$ verify, in some order, \eqref{yau} for $c>0$.
\end{theorem}
\begin{proof}
Let $u(r)$ be defined by

$$\left\{\def\arraystretch{1.5}\begin{array}{ll} u(r)=0 & \text{ if $0\leq r \leq 1$}, \\ u(r) =  - \int_1^r \sqrt{\frac{c^3 (r^{2/3}-1)^3}{1-c^3 (r^{2/3}-1)^3}} \, dr & \text{ if $1\leq r < (1+1/c)^{3/2}.$}\end{array} \right.$$

A computation shows that $u$ is of class $C^2$, and extends continuously with a finite value to $r_0:=(1+1/c)^{3/2}$. Also, $u'(r)\to -\8$ as $r\to r_0$. Let $S^+$ be the rotational $C^2$ surface given by $z=u(r)$, with $r=\sqrt{x^2+y^2}$, for $0\leq r \leq r_0$. The principal curvatures of $S^+$ are zero if $r\leq 1$, and are given (for the downwards unit normal) by $$\landa(r)= \left(c - \frac{c}{r^{2/3}}\right)^{3/2}, \hspace{1cm} \mu(r) = c \landa(r)^{1/3},$$ if $1\leq r \leq r_0.$ Note that $\mu(r_0),\landa(r_0)$ are finite and the tangent plane of $S^+$ is vertical along $\parc S^+$. In this way, the surface $S^+$ is $C^2$-smooth up to its boundary. Moreover, writing $\kappa_1=\mu(r)$, $\kappa_2=\landa(r)$, we see that $S^+$ satisfies \eqref{yau}.

Let $S^-$ denote the reflection of $S^+$ with respect to the horizontal plane of $\R^3$ given by $z=u(r_0)$, and consider the union $S=S^+\cup S^-$. It is clear that $S$ is a compact $C^2$ surface diffeomorphic to $\S^2$, which is not an ellipsoid and satisfies \eqref{yau}. 
\end{proof}

\begin{remark}
The existence of the above example was overlooked in \cite{KS,Si}, where Yau's question was studied for the particular case of rotational surfaces. In particular, our example contradicts Proposition 10 in \cite{KS}. 
\end{remark}

The basic idea behind the construction in Theorem \ref{th:yau} is that one can consider the totally umbilic rotational example that starts at the rotation axis (in this case, a plane), and then bifurcate from it away from the axis. This behavior is not possible for elliptic Weingarten surfaces. Indeed, Gálvez and Mira proved in \cite{GM3} that the umbilics of any surface in $\R^3$ satisfying an elliptic Weingarten equation $\kappa_2 =g(\kappa_1)$, with $g'(0)<0$, are necessarily isolated. (For the particular \emph{symmetric} case in which $g'(0)=-1$, this is a classical theorem, see \cite{HW1} or \cite{B}). 

However, if we slightly weaken the ellipticity condition $g'<0$ to the degenerate elliptic case $g'\leq 0$, the umbilics of a Weingarten surface satisfying $\kappa_2=g(\kappa_1)$ fail to be isolated in general, see \cite{Mu,GMT}. Specifically, there exist examples of such degenerate elliptic Weingarten equations where one can bifurcate from a totally umbilic sphere meeting the rotation axis, to create non-round, degenerate elliptic, rotational Weingarten spheres. See Figure A.1 in \cite{GMT}.

\def\refname{References}

\vskip 0.2cm

\noindent Isabel Fernández

\noindent Departamento de Matemática Aplicada I,\\ Instituto de Matemáticas IMUS \\ Universidad de Sevilla (Spain).

\noindent  e-mail: {\tt isafer@us.es}

\vskip 0.2cm

\noindent Pablo Mira

\noindent Departamento de Matemática Aplicada y Estadística,\\ Universidad Politécnica de Cartagena (Spain).

\noindent  e-mail: {\tt pablo.mira@upct.es}

\vskip 0.4cm

\noindent This research has been financially supported by Project PID2020-118137GB-I00 funded by MCIN/AEI /10.13039/501100011033.


\begin{thebibliography}{9}

\bibitem{AEG} J.A. Aledo, J.M. Espinar, J.A. Gálvez, The Codazzi equation for surfaces, {\it Adv. Math.} {\bf 224} (2010), 2511--2530.

\bibitem{A} A.D. Alexandrov, Uniqueness theorems for surfaces
in the large, I, {\it Vestnik Leningrad Univ.} {\bf 11} (1956),
5--17. (English translation: {\it Amer. Math. Soc. Transl.}  {\bf 21} (1962), 341--354).

\bibitem{B} R. Bryant, Complex analysis and a class of Weingarten surfaces. preprint. arXiv:1105.5589.

\bibitem{BGM} A. Bueno, J.A. Gálvez, P. Mira, Rotational hypersurfaces of prescribed mean curvature, {\it J. Diff. Equations} {\bf 268} (2020), 2394--2413.

\bibitem{BO} A. Bueno, I. Ortiz, Surfaces of prescribed linear Weingarten curvature in $\R^3$, preprint (2022), arXiv:2201.07480.


\bibitem{CC} P. Carretero, I. Castro, A new approach to rotational Weingarten surfaces, preprint (2021), arXiv:2112.15395

\bibitem{Ch0}S.S. Chern, Some new characterizations of the Euclidean sphere, {\it Duke Math. J.} {\bf 12} (1945), 279--290.

\bibitem{Ch} S.S. Chern, On special $W$-surfaces.
{\it Proc. Amer. Math. Soc.} {\bf 6} (1955), 783--786. 

\bibitem{CF}A.V. Corro, W. Ferreira, K. Tenenblat, Ribaucour transformations for constant mean curvature and linear Weingarten surfaces, {\it Pacific J. Math.} {\bf 212} (2003), 265--297.

\bibitem{ER} R. Sa Earp, H. Rosenberg, Some remarks on surfaces of prescribed mean curvature. In: Differential Geometry, pp. 123--148. Pitman Monogr. Surveys Pure Appl. Math., no. 52. Longman Sci. Tech., Harlow, 1991.

\bibitem{EM} J.M. Espinar, H. Mesa, Elliptic special Weingarten surfaces of minimal type in $\R^3$ of finite total curvature, preprint (2019). arXiv:1907.09122

\bibitem{FGM} I. Fernández, J.A. Gálvez, P. Mira, Quasiconformal Gauss maps and the Bernstein problem for Weingarten multigraphs, preprint (2020). Arxiv, 2004.08275

\bibitem{GHM}J.A. Gálvez, L. Hauswirth, P. Mira, Surfaces of constant curvature in $\R^3$ with isolated singularities, {\it Adv. Math.} {\bf 241} (2013) 103--126.

\bibitem{GJM}J.A. Gálvez, A. Jiménez, P. Mira, Isolated singularities of graphs in warped products and Monge-Ampère equations, {\it J. Differential Equations} {\bf 260} (2016), 2163--2189.

\bibitem{GMM} J.A. Gálvez, A. Martínez, F. Milán, Linear Weingarten surfaces in $\R^3$, {\it Monatsh. Math} {\bf 138} (2003), 133--144.

\bibitem{GM3} J.A. Gálvez, P. Mira, Uniqueness of immersed spheres in three-manifolds, {\it J. Diff. Geom.} {\bf 116} (2020), 459--480.

\bibitem{GM4} J.A. Gálvez, P. Mira, Rotational symmetry of Weingarten spheres in homogeneous three-manifolds, {\it J. Reine Angew. Math.} {\bf 773} (2021), 21--66.

\bibitem{GM5} J.A. Gálvez, P. Mira, Serrin's overdetermined problem for fully nonlinear, non-elliptic equations, {\it Anal. PDE}, {\bf 14} (2021), 1429--1442.


\bibitem{GMT} J.A. Gálvez, P. Mira, M.P. Tassi, A quasiconformal Hopf soap bubble theorem, preprint (2021), arXiv:2103.12665

\bibitem{HW1} P. Hartman, A. Wintner, Umbilical points and $W$-surfaces, {\it Amer. J. Math.} {\bf 76} (1954), 502--508.

\bibitem{HM} D. Hoffman, W.H. Meeks, The strong halfspace theorem for minimal surfaces, {\it Invent. Math.} {\bf 101} (1990), 373--377.

\bibitem{Ho0} H. Hopf, Uber Flachen mit einer Relation zwischen den Hauptkrummungen, {\it Math. Nachr.} {\bf 4} (1951), 232--249.

\bibitem{Ho} H. Hopf. {\it Differential Geometry in the Large, volume 1000 of
Lecture Notes in Math.} Springer-Verlag, 1989.

\bibitem{JS} A. Jiménez, J.P dos Santos, Isolated singularities of elliptic linear Weingarten graphs, \emph{preprint}, 2021.

\bibitem{KS} W. Kühnel, M. Steller, On closed Weingarten surfaces, {\it Monatsh. Math.} {\bf 146} (2005), 113--126.

\bibitem{rafa} R. López, A. Pámpano, Classification of rotational surfaces in Euclidean space satisfying a linear relation between their principal curvatures, {\it Math. Nachr.} {\bf 293} (2020), 735--753.

\bibitem{MO1} I.M. Mladenov, J. Oprea, The mylar balloon revisited, {\it Amer. Math. Monthly}, {\bf 110} (2003), 761--784.

\bibitem{MO2} I.M. Mladenov, J. Oprea, The mylar balloon: new viewpoints and generalizations. In Geometry, Integrability and Quantization, VIII. Sofia, 2007, pp. 246--263.

\bibitem{Mu} H.F. Münzner, Über Flächen mit einer Weingartenschen Ungleichung, {\it Math. Z.} {\bf 97} (1967), 123--139.

\bibitem{P} B. Papantoniou, Classification of the surfaces of revolution whose principal curvatures are connected by the relation
$A\kappa_1 +B \kappa_2=0$, where $A$ or $B$ is different of from zero, {\it Bull. Calcutta Math. Soc.} {\bf 76} (1984) 49--56.

\bibitem{RaS} Th.M.Rassias, R.Sa Earp, Some problems in analysis and geometry. Complex Analysis in Several Variables, Hadronic Press, Florida 1999, pp. 111--122.

\bibitem{RS} H. Rosenberg, R. Sa Earp, The geometry of properly embedded special surfaces in $\R^3$; e.g. surfaces satisfying $aH+bK=1$, where $a$ and $b$ are positive, {\it Duke Math. J.} {\bf 73} (1994), 291--306.

\bibitem{SaTo} R. Sa Earp, E. Toubiana, Classification des surfaces de type Delaunay, {\it Amer. J. Math.} {\bf 121} (1999), 671--700.

\bibitem{SaTo2} R. Sa Earp, E. Toubiana, Sur les surfaces de Weingarten spéciales de type minimal, {\it Bull. Braz. Math. Soc.}  {\bf 26} (1995), 129--148.

\bibitem{Sch} R. Schoen, Uniqueness, symmetry and embeddedness of minimal surfaces, {\it J. Diff. Geom.} {\bf 18} (1983), 791--809.

\bibitem{Si} U. Simon, Yau's problem on a characterization of rotational ellipsoids, {\it Asian J. Math.} {\bf 11} (2007), 361--372.

\bibitem{T} K. Tenenblat, Transformations on manifolds and applications to differential equations, Pitman Monographs and Surveys in Pure and Applied Mathematics, Longman, Harlow, 1998.

\bibitem{Vo} K. Voss, Uber geschlossene Weingartensche Flachen, {\it Math. Ann.} {\bf 138} (1959), 42--54.

\bibitem{Yau} S.T. Yau, Problem Section. Seminar on Differential Geometry. Annals of Mathematical Studies, No. 102, Princeton University Press, 1982, pp. 669--706.

\end{thebibliography}
\end{document}